\documentclass[11pt]{article}
\usepackage[utf8]{inputenc}
\usepackage[english]{babel}
\usepackage[none]{hyphenat}

\usepackage{amsmath}
\usepackage{amssymb}
\usepackage{amsthm}

\usepackage{amsfonts}

\usepackage{setspace}
\usepackage{fullpage}
\usepackage{graphicx}
\usepackage{url}
\usepackage{indentfirst}

\usepackage{xcolor}
\usepackage{graphicx}

\usepackage{float}

\usepackage{textcomp}
\usepackage{gensymb}
\usepackage{comment}

\usepackage{subcaption}

\usepackage{enumerate}

\usepackage{soul}
\usepackage{xcolor}

\usepackage{tikz-cd}

\usepackage{ dsfont }

\newtheorem{innercustomthm}{Theorem}
\newenvironment{customthm}[1]
  {\renewcommand\theinnercustomthm{#1}\innercustomthm}
  {\endinnercustomthm}

\newtheorem{theorem}{Theorem}[section]

\newtheorem{lemma}[theorem]{Lemma}
\newtheorem{prop}[theorem]{Proposition}

\theoremstyle{definition}
\newtheorem{definition}[theorem]{Definition}
\newtheorem{example}[theorem]{Example}

\newtheorem{remark}[theorem]{Remark}

\theoremstyle{definition}

\usepackage[capitalize]{cleveref}
	\crefname{app-corollary}{Corollary}{Corollaries}
	\Crefname{app-corollary}{Corollary}{Corollaries}
	\crefname{app-definition}{Definition}{Definitions}
	\Crefname{app-definition}{Definition}{Definitions}
	\crefname{figure}{Figure}{Figures}
	\Crefname{figure}{Figure}{Figures}
	\crefname{lemma}{Lemma}{Lemmata}
	\Crefname{lemma}{Lemma}{Lemmata}
	\crefname{app-lemma}{Lemma}{Lemmata}
	\Crefname{app-lemma}{Lemma}{Lemmata}
	\crefname{app-proposition}{Proposition}{Proposition}
	\Crefname{app-proposition}{Proposition}{Proposition}
	\crefname{app-theorem}{Theorem}{Theorems}
	\Crefname{app-theorem}{Theorem}{Theorems}

\newcommand\R{\mathbb{R}}
\newcommand\N{\mathbb{N}}

\newcommand{\set}[1]{\left\{ #1 \right\} }

\newcommand{\given}{\;|\;}

\renewcommand{\epsilon}{\varepsilon}

\definecolor{highlight}{HTML}{cdeef2}

\newcommand{\Ass}{\text{Ass}}
\newcommand{\Spec}{\text{Spec}}
\newcommand{\Min}{\text{Min}}
\renewcommand{\dim}{\mathrm{dim}}

\title{Controlling Formal Fibers of \\Countably Many Principal Prime Ideals}
\author{David Baron, Ammar Eltigani, S. Loepp, AnaMaria Perez, M. Teplitskiy}
\date{\today}

\begin{document}

\setlength{\parindent}{10 pt}
\setlength{\abovedisplayskip}{2pt}
\setlength{\belowdisplayskip}{2pt}
\setlength{\abovedisplayshortskip}{0pt}
\setlength{\belowdisplayshortskip}{0pt}

\setlength\fboxsep{10pt}

\maketitle


\begin{abstract}
    Let $T$ be a complete local (Noetherian) ring.  For each $i \in \mathbb{N}$, let $C_i$ be a nonempty countable set of nonmaximal pairwise incomparable prime ideals of $T$, and suppose that if $i \neq j$, then either $C_i = C_j$ or no element of $C_i$ is contained in an element of $C_j$. We provide necessary and sufficient conditions for $T$ to be the completion of a local integral domain $A$ satisfying the condition that, for all $i \in \mathbb{N}$, there is a nonzero prime element $p_i$ of $A$, such that $C_i$ is exactly the set of maximal elements of the formal fiber of $A$ at $p_iA$. We then prove related results where the domain $A$ is required to be countable and/or excellent.
\end{abstract}

\section{Introduction}


Completions of local (Noetherian) rings have proven to be a useful tool in commutative algebra. Unfortunately, not all aspects of the relationship between a local ring and its completion with respect to its maximal ideal are well understood. In this paper, we are interested in gaining a better understanding of the relationship between the prime ideals of a local ring and the prime ideals of its completion. 

Recall that the formal fiber of a local ring $A$ at a prime ideal $P$ of $A$ is defined to be $\Spec(T \otimes_A k(P))$ where $k(P)= A_P/PA_P$ and $T$ is the completion of $A$ with respect to its maximal ideal. We note that there is a one-to-one correspondence between the formal fiber of $A$ at $P$ and the prime ideals $Q$ of $T$ satisfying $Q \cap A = P$.  Thus, it is often useful to think of the formal fiber of $A$ at $P$ as the set of prime ideals of the completion of $A$ that lie over $P$, and this is how we will think of the formal fiber of $A$ at $P$ for the remainder of this paper. Since $T$ is a faithfully flat extension of $A$, the map $\Spec(T) \longrightarrow \Spec(A)$ given by $Q \longrightarrow Q \cap A$ is onto. It follows that every prime ideal of $T$ is in the formal fiber of some prime ideal of $A$. Hence, the formal fibers of a local ring $A$ form a partition of the prime ideals of $T$. We aim to understand which partitions of the prime ideals of $T$ are possible. More specifically, we are interested in the following question.

\medskip

\noindent {\bf Question:} Let $T$ be a complete local ring and let $\mathcal{C} = \{C_{\alpha}\}_{\alpha \in \Omega}$ be a partition of $\Spec(T)$. Under what conditions does there exist a local ring $A$ such that the completion of $A$ with respect to its maximal ideal is $T$ and such that, for all $C_{\alpha} \in \mathcal{C}$, there exists a prime ideal $P_{\alpha}$ of $A$ such that $C_{\alpha}$ is exactly the formal fiber of $A$ at $P_{\alpha}$?

\medskip

\noindent In other words, we ask, given a complete local ring $T$, under what conditions is it possible to find a local ring $A$ with completion $T$ such that all formal fibers of $A$ are controlled? We believe that answering this question is very difficult, and so we focus on controlling formal fibers of specific prime ideals of $A$. For example, one could ask how well the formal fibers of minimal prime ideals can be controlled.  The results in \cite{Arnosti} give insight into the answer to this question.  In this paper, we are interested in how well formal fibers of height one prime ideals can be controlled, furthering previous work on the topic (see, for example, \cite{Chatlos} and \cite{Boocher}).
In particular, in \cite{Chatlos}, it is shown that the formal fiber of exactly one height one prime ideal can be controlled. 
 We extend this result by showing that the formal fibers of countably many height one prime ideals can be controlled.
Specifically, in Section \ref{uncountable section}, we prove the following result.
\begin{customthm}{\ref{big theorem}}
Let $T$ be a complete local ring and let $\Pi$ denote the prime subring of $T$. For each $i \in \mathbb{N}$, let $C_i$ be a nonempty countable set of nonmaximal pairwise incomparable prime ideals of $T$ and suppose that, if $i \neq j$, then either $C_i = C_j$ or no element of $C_i$ is contained in an element of $C_j$. Then $T$ is the completion of a local domain $A \subseteq T$ satisfying the condition that, for all $i \in \mathbb{N}$, there is a nonzero prime element $p_i$ of $A$ such that $C_i$ is exactly the set of maximal elements of the formal fiber of $A$ at $p_iA$ if and only if there exists a set of nonzero elements $\mathfrak{q} = \{q_i\}_{i = 1}^{\infty}$ of $T$ satisfying the following conditions
\begin{enumerate}[(i)]
        \item For $i \in \mathbb{N}$ we have $q_i \in \bigcap_{Q \in C_i}Q$ and if $C_j \neq C_i$ and $Q' \in C_j$, then $q_i \not\in Q'$,
        \item $P \cap \Pi[\mathfrak{q}]=(0)$ for all $P \in \Ass(T)$,
        \item If $i \in \mathbb{N}$ and $P' \in \Ass(T/q_iT)$, then $P' \subseteq Q$ for some $Q \in C_i$, and
        \item If $i \in \mathbb{N}$ and $Q \in C_i$, then $F_{\Pi[\mathfrak{q}]}\cap Q \subseteq q_iT$ where $F_{\Pi[\mathfrak{q}]}$ is the quotient field of $\Pi[\mathfrak{q}]$.
    \end{enumerate}
\end{customthm}

\noindent Showing that the above four conditions are sufficient is much more difficult than showing that they are necessary, and so, the bulk of our work will be showing that the conditions are sufficient. To do this, we start with a complete local ring $T$, the sets $C_i$ satisfying our hypotheses, and a set of elements $\mathfrak{q} = \{q_i\}_{i = 1}^{\infty}$ of $T$ satisfying the four conditions. We then construct a local domain $A$ that satisfies the properties given in Theorem \ref{big theorem}. We define the notion of a $Ct\mathfrak{q}$-subring of $T$ which will aid in our construction of $A$. We end Section 2 by showing these conditions are in fact necessary, thereby completing the proof of our theorem. In Section 3, we show that we can have some control over formal fibers of countably many height one prime ideals of \textit{countable} domains, and in
Section 4, we show that we can have some control over formal fibers of countably many height one prime ideals of quasi-excellent and excellent domains. 
In Section 5, we combine the results from Section 3 and Section 4 to show that we can have some control over formal fibers of countably many height one prime ideals of countable quasi-excellent domains and of countable excellent domains.

Throughout the paper, $\N$ will be assumed to be the positive integers. Moreover, we say a ring $R$ is quasi-local if it has unique maximal ideal $M$, whilst $R$ is local if it is both quasi-local and Noetherian. We let $(R,M)$ denote a quasi-local ring $R$ with maximal ideal $M$, and we denote the completion of a local ring $R$ with respect to its maximal ideal by $\widehat{R}$. We call a local ring $A$ a \textit{precompletion} of a complete local ring $T$ if $\widehat{A}\cong T$. When we say that $C$ is a set of incomparable (or pairwise incomparable) prime ideals of a ring $R$, we mean that for all pairs of prime ideals $P,P' \in C$, we have $P \not\subseteq P'$. Finally, if $R$ is an integral domain, we use $F_R$ to donote the quotient field of $R$.

\section{The Main Theorem}\label{uncountable section}
The goal of this section is to prove Theorem \ref{big theorem}. As mentioned in the previous section, most of our work is dedicated to showing that the four conditions given in Theorem \ref{big theorem} are sufficient. To show that they are sufficient, we start with a complete local ring $T$, the sets $C_i$, and the set of elements $\mathfrak{q}$ of $T$ satisfying the four conditions of the theorem. We then adjoin the set $\mathfrak{q}$ to the prime subring of $T$. The next step is to carefully successively adjoin elements to this ring to create an ascending chain of subrings of $T$ with each subring satisfying very specific properties. The union of these subrings will be the desired precompletion of $T$. 

\subsection{Preliminaries}
We begin with preliminary results that will help with the construction of our precompletion. Cardinality arguments will play a central role in our construction. The following proposition, taken from \cite{Dundon}, will be used for some of these cardinality arguments. 


\begin{prop} \label{Yu 2.10}
    (\cite{Dundon}, Lemma 2.2) Let $(T,M)$ be a complete local ring with $\dim(T) \geq 1$. Let $P$ be a nonmaximal prime ideal of $T$. Then, $|T/P| = |T| \geq c$, where $c$ denotes the cardinality of $\R$.
\end{prop}

Recall that, given a complete local ring $T$, one of the conditions we want our local domain $A$ to satisfy is that its completion is $T$.  The following proposition from \cite{HEITMANN1994} provides conditions for a quasi-local subring of $T$ to be Noetherian and have $T$ as its completion. 

\begin{prop} \label{completion proving machine}
    (\cite{HEITMANN1994}, Proposition 1)  If $(R, R \cap M )$ is a quasi-local subring of a complete local ring $(T,M)$, the map $R \rightarrow T/M^2$ is onto, and $IT \cap R = I$ for every finitely generated ideal $I$ of $R$, then $R$ is Noetherian and the natural homomorphism $\widehat{R} \rightarrow T$ is an isomorphism.
\end{prop}
 
Our final local domain $A$ will be a subring of $T$ that satisfies the conditions of Proposition \ref{completion proving machine}. To ensure that $A$ satisfies these conditions, we repeatedly use the following result from \cite{Vu}. Note that Proposition \ref{stronger coset avoidance} can be thought of as a generalization of the countable prime avoidance theorem for complete local rings.


\begin{prop}\label{stronger coset avoidance}
    (\cite{Vu}, Lemma 2.7) Let $(T,M)$ be a complete local ring such that $\dim (T) \geq 1,$ let $C$ be a countable set of incomparable nonmaximal prime ideals of $T$ and let $D$ be a subset of $T$ such that $|D| < |T|$. Let $I$ be an ideal of $T$ such that $I\nsubseteq P$ for all $P\in C$. Then $I\nsubseteq \bigcup\set{r+P\given r\in D,P\in C}$.
\end{prop}


The next four lemmas are useful results regarding subrings of Noetherian rings where the subrings are not assumed to be Noetherian.

\begin{lemma}\label{factoring}
    Let $T$ be a ring and let $R$ be a subring of $T$ that contains no zerodivisors of $T$. Let $q$ be a nonzero element of $R$ such that $qT\cap R = qR$. Then for every $\ell \in \mathbb{N}$, we have $(qT)^\ell\cap R = (qR)^\ell$. 
\end{lemma}
\begin{proof}
    We induct on $\ell$. By hypothesis, the statement holds for $\ell = 1$. Assume $(qT)^\ell\cap R = (qR)^\ell$. Let $x \in (qT)^{\ell + 1}\cap R$ and write $x=q(q^{\ell}t)$ for some $t \in T$. Then $x \in qT \cap R = qR$, and so $x = qr$ for some $r \in R$. Thus $q(q^{\ell}t) = qr$ and since $q$ is not a zerodivisor, we have $r = q^{\ell}t \in (qT)^\ell\cap R = (qR)^\ell$. It follows that $x = q(q^{\ell}t) \in (qR)(q R)^\ell = (q R)^{\ell+ 1}$. Therefore, $(qT)^{\ell+ 1}\cap R \subseteq (qR)^{\ell+ 1}$.  Since the reverse inclusion holds, the result follows.
\end{proof}


\begin{lemma}\label{largestk}
Let $(T,M)$ be a local ring and let $R$ be a subring of $T$.  Suppose $q$ and $a$ are elements of $R$ such that $a \neq 0$ and $q \in M$. If $a\in qR$, then there is a positive integer $k$ such that $a\in q^kR$ and $a\not\in q^{k+ 1}R$.
\end{lemma}

\begin{proof}
Assume that no such postive integer exists.  Then $a \in \bigcap_{n = 1}^{\infty} (qR)^n \subseteq \bigcap_{n = 1}^{\infty} (qT)^n \subseteq \bigcap_{n = 1}^{\infty} (M)^n.$ Since $T$ is Noetherian, Krull's intersection theorem gives us that $\bigcap_{n = 1}^{\infty} (M)^n = (0)$.  It follows that $a = 0$, a contradiction.
\end{proof}

\begin{lemma}\label{krull's intersection}
    Let $(T,M)$ be a local ring and let $(R,R \cap M)$ be a quasi-local subring of $T$ that contains no zerodivisors of $T$. If $y$ is a nonzero element of $R$, then $y$ is contained in at most finitely many principal prime ideals of $R$.
\end{lemma}
\begin{proof}
Let $y \in R$, and suppose $\{q_iR\}_{i \in \mathbb{N}}$ are prime ideals of $R$ with $y \in q_iR$ for all $i \in \mathbb{N}$ and with $q_iR = q_jR$ if and only if $i = j$. We claim that, for $k \in \mathbb{N}$, there is an element $r_k$ in $R$ such that $y = (q_1q_2\cdots q_k)r_k$. To show this, we induct on $k$. Since $y \in q_1R$, there is an $r_1 \in R$ such that $y = q_1r_1$.  Now suppose $k \geq 1$ and assume there is an $r_k \in R$ such that $y = (q_1q_2\cdots q_k)r_k$. Now $y \in q_{k + 1}R$, and since $q_{k + 1}R$ is a prime ideal of $R$, we have that $r_k \in q_{k + 1}R$ or $q_i \in q_{k + 1}R$ for some $i = 1,2, \ldots ,k$. If $r_k \in q_{k + 1}R$, then $r_k = q_{k + 1}r_{k + 1}$ for some $r_{k + 1} \in R$.  Hence, $y = (q_1q_2\cdots q_kq_{k + 1})r_{k+ 1}$, and our claim holds by induction. So suppose $q_i \in q_{k + 1}R$ for some $i = 1,2, \ldots ,k$. Then $q_iR \subseteq q_{k + 1}R$ and so $q_i = q_{k + 1}r$ for some $r \in R$. Since $q_iR$ is a prime ideal of $R$, $q_{k + 1} \in q_iR$ or $r \in q_iR$. If $q_{k + 1} \in q_iR$ then $q_iR = q_{k + 1}R$, a contradiction. If $r \in q_iR$ then $r = q_ir'$ for some $r' \in R$ and we have $q_i = q_{k + 1}q_ir'$. Since $q_i$ is not a zerodivisor in $T$, we have $1 = q_{k + 1}r'$, and so $q_{k + 1}$ is a unit, contradicting that $q_{k + 1}R$ is a prime ideal of $R$. Thus our claim holds.  It follows that $y \in \bigcap_{k \in \mathbb{N}} (M)^k$,  and so by Krull's intersection theorem, $y = 0$.
%
%
\end{proof}

If $R$ is a subring of the ring $T$ and $Q$ is an ideal of $T$, then there is an injective map $R/(R \cap Q) \longrightarrow T/Q$. Therefore, it makes sense to say that an element $t + Q \in T/Q$ is either algebraic or transcendental over the ring $R/(R \cap Q)$.

\begin{lemma} \label{transcendentals}
Let $(T,M)$ be a local ring and let $R$ be a subring of $T$. Let $q$ be an element of $R$ such that $q \in M$ and $q$ is a regular element of $T$.  Suppose $Q$ is an ideal of $T$ such that  $R\cap Q = qR$. If $t \in T$ with $t+Q\in T/Q$ transcendental over $R/(R\cap Q)$, then $t$ is transcendental over $R$.
\end{lemma}

\begin{proof}

Suppose $t\in T$ is algebraic over $R$. Then $a_0 + a_1t + a_2t^2 +\dots+ a_nt^n = 0$ for some $a_i \in R$ with $a_n \neq 0$. Since $t+Q\in T/Q$ is transcendental over $R/(R\cap Q)$, we have $a_i \in R \cap Q = qR$ for all $i \in \{0,1, \ldots ,n\}$. If $a_i \neq 0,$ let $k_i$ be the largest positive integer such that $a_i\in q^{k_i}R$. Note that $k_i$ exists by Lemma \ref{largestk}.  
Let $j = \min\set{k_i}$. Since $j \leq k_i$ for all $i$,
$$0 = a_0 + a_1t + a_2t^2 +\dots+ a_nt^n = q^j(b_0 + b_1t + b_2t^2 +\dots+ b_nt^n)$$
for some $b_i \in R$ where at least one $b_i$ is not in $qR$. Since $q$ is a regular element of $T$, we have $b_0 + b_1t + b_2t^2 +\dots+ b_nt^n = 0$. Since $t+Q\in T/Q$ is transcendental over $R/(R\cap Q)$, we have $b_i \in R \cap Q = qR$ for all $i$, a contradiction. 
\end{proof}



In this paper, we generalize the main results of \cite{Chatlos}. In particular, in \cite{Chatlos}, it is shown that, given a complete local ring $T$ along with a finite set of prime ideals $C$ of $T$, there are conditions for which $T$ has a precompletion $A$ such that $A$ has a principal prime ideal whose formal fiber's maximal elements are exactly the elements of $C$.  We generalize this result in two ways.  First, we show that the set $C$ can be infinitely countable.  Second, we find conditions for which $T$ has a precompletion $A$ such that $A$ has countably many principal prime ideals that have countably many maximal elements of their formal fibers.  Moreover, this set of maximal elements can be controlled. 
%
Many of our results and the proofs for our results are analogous to those from \cite{Chatlos}. Therefore, before stating our own results, we first state the following important definition from \cite{Chatlos}.

\begin{definition}[\cite{Chatlos}, Definition 2.2] \label{Chatlos 2.2}
     Let $(T,M)$ be a complete local ring, and suppose we have a finite, pairwise incomparable set $C = \{Q_1, Q_2,\dots ,Q_k\} \subseteq \Spec (T)$. Let $p \in \bigcap_{i=1}^k Q_i$ be a nonzero regular element of $T$. Suppose that $(R,R \cap M)$ is a quasi-local subring of $T$ containing $p$ with the following properties:
    \begin{enumerate}[(i)]
        \item $\Gamma(R) < |T|$, where $\Gamma(R)=\sup\{ |\N|, |R|\}$,
        \item If $P$ is an associated prime ideal of $T$ then $R \cap P = (0)$, and
        \item For all $i \in \{1,\dots,k\}$, $F_R \cap Q_i \subseteq pT$.
    \end{enumerate}
    Then we call $R$ a \textit{$p$ca-subring of T}.
\end{definition}



\subsection{The Construction}
We introduce the definition of a $Ct\mathfrak{q}$-subring of a complete local ring $T$, an essential component in our construction. Note that $Ct\mathfrak{q}$-subrings are generalizations of $p$ca-subrings from \cite{Chatlos}. 

\begin{definition}
\label{Definition 6.1}
    Let $(T,M)$ be a complete local ring. For each $i \in \mathbb{N}$, let $C_i$ be a nonempty countable set of nonmaximal pairwise incomparable prime ideals of $T$ and such that, if $i \neq j$, then either $C_i = C_j$ or no element of $C_i$ is contained in an element of $C_j$. For each $i \in \mathbb{N}$, let $q_i$ be a nonzero regular element of $T$ such that $q_i \in \bigcap_{Q \in C_i}Q$ and, if $C_j \neq C_i$ and $Q' \in C_j$, then $q_i \not\in Q'$. Suppose that $(R, R\cap M)$ is an infinite quasi-local subring of $T$ containing $q_i$ for all $i \in \mathbb{N}$ such that:
    \begin{enumerate}[(i)]
        \item $|R| < |T|$,
        \item If $P \in \Ass(T)$, then $R \cap P = (0)$,  and
        \item If $i \in \mathbb{N}$, then $F_R \cap Q \subseteq q_iT$ for all $Q\in C_i$.
    \end{enumerate}
    Let $\mathfrak{q} = \{q_i \given i \in \mathbb{N}\}$ and let $C = \{C_i \given i \in \mathbb{N}\}$. Then we call $R$ a \textit{C to} $\mathfrak{q}$ \textit{subring of T}, or a \textit{$Ct\mathfrak{q}$-subring of $T$} for short.
\end{definition}

When the ring $T$ is understood, we will often call $R$ a $Ct\mathfrak{q}$-subring instead of a $Ct\mathfrak{q}$-subring of $T$. Notice that condition $(ii)$ implies that $Ct\mathfrak{q}$-subrings are integral domains. Moreover, $Ct\mathfrak{q}$-subrings contain regular elements in $T$ of the form $q_i$ for all $i \in \mathbb{N}$. These elements will be used to generate principal prime ideals $p_iA$ in our final local domain $A$ such that the formal fiber of $A$ at $p_iA$ has maximal elements exactly the elements of $C_i$. We note that if $T$ is a complete local ring that contains a $Ct\mathfrak{q}$-subring, then $T$ has prime ideals that are not maximal that contain regular elements of $T$. It follows that the dimension of $T$ is at least two.

In order to construct the desired local domain $A$ with completion $T$, we adjoin elements to $Ct\mathfrak{q}$-subrings and then localize. By doing this, we construct a chain of $Ct\mathfrak{q}$-subrings and show that their union is our desired precompletion $A$. We next show that, under certain circumstances, a localization of a $Ct\mathfrak{q}$-subring is a $Ct\mathfrak{q}$-subring, and the union of a chain of $Ct\mathfrak{q}$-subrings is a $Ct\mathfrak{q}$-subring. These results help us maintain the $Ct\mathfrak{q}$-subring properties throughout most of our construction. We note that Lemma \ref{localization lemma} is the analog of Lemma 2.4 from \cite{Chatlos} and Lemma \ref{unioning lemma} is the analog of Lemma 2.5 from \cite{Chatlos}, and the proofs of Lemmas \ref{localization lemma} and \ref{unioning lemma} are largely based on the proofs of those lemmas in \cite{Chatlos}.

\begin{lemma} \label{localization lemma}
   Let $(T,M)$, $C_i$, and $q_i$ for $i \in \mathbb{N}$ be as in Definition \ref{Definition 6.1}. Suppose $R$ is an infinite subring of $T$. If $R$ satisfies conditions (i)-(iii) from Definition \ref{Definition 6.1}, then $R_{(R\cap M)}$ satisfies conditions (i)-(iii) from Definition \ref{Definition 6.1}.  In particular, if $R$ contains $q_i$ for all $i \in \mathbb{N}$, then $R_{(R\cap M)}$ is a $Ct\mathfrak{q}$-subring of $T$.
\end{lemma}
\begin{proof}
    Let $R$ be an infinite subring of $T$ satisfying properties $(i)$, $(ii)$, and $(iii)$ from Definition \ref{Definition 6.1}. Since $R$ is infinite, we have $|R_{(R \cap M)}| = |R| < |T|$. Now suppose $P \in \Ass(T)$ and let $x \in R_{(R \cap M)} \cap P$. Then $r_1x = r_2$ for $r_1,r_2 \in R$ with $r_1 \neq 0$. Thus $r_2 \in R \cap P = (0)$ and it follows that $x = 0$.  Hence, condition $(ii)$ is satisfied. Finally, let $i \in \mathbb{N}$ and let $Q \in C_i$. Suppose $y \in F_{R_{(R \cap M)}} \cap Q$. Then $y \in F_R \cap Q \subseteq q_iT$ and so condition $(iii)$ is satisfied. Finally, note that $R_{(R\cap M)}$ is an infinite quasi-local subring of $T$. If $R$ contains $q_i$ for all $i$, then so does $R_{(R \cap M)}$ and so $R_{(R \cap M)}$ is a $Ct\mathfrak{q}$-subring of $T$.
\end{proof}

A main part of our construction is to build a chain of $Ct\mathfrak{q}$-subrings of $T$ and then claim that the union of the $Ct\mathfrak{q}$-subrings in the chain satisfies most properties of being a $Ct\mathfrak{q}$-subring of $T$.  We use the next lemma to prove this claim.

\begin{lemma} \label{unioning lemma}
    Let $(T,M)$, $C_i$, and $q_i$ for $i \in \mathbb{N}$ be as in Definition \ref{Definition 6.1}. Let $\Omega$ be a well-ordered set, and let $\{(R_{\alpha}, R_{\alpha} \cap M) \given \alpha \in \Omega\}$ be a set of $Ct\mathfrak{q}$-subrings of $T$ indexed by $\Omega$ with the property that $R_{\alpha} \subseteq R_{\beta}$ for all $\alpha$ and $\beta$ such that $\alpha < \beta$. Let $S = \bigcup_{\alpha \in \Omega}R_{\alpha}$. Then $S$ is quasi-local with maximal ideal $S \cap M$ and satisfies conditions (ii), and (iii) of Definition \ref{Definition 6.1}. Furthermore, if, for some $\lambda$, $|R_{\alpha}| \leq \lambda$ for all $\alpha \in \Omega$, we have $|S| \leq \lambda |\Omega|$. So, if $|\Omega| \leq \lambda$ and $|R_{\alpha}|=\lambda$ for some $\alpha$, $|S|=\lambda$.
\end{lemma}
\begin{proof}
    The cardinality statement does not require a proof. Suppose $\{(R_{\alpha}, R_{\alpha} \cap M) \,| \,\alpha \in \Omega\}$ is a set of $Ct\mathfrak{q}$-subrings indexed by $\Omega$ with the property $R_{\alpha} \subseteq R_{\beta}$ for all $\alpha$ and $\beta$ such that $\alpha < \beta$. Since each $R_{\alpha}$ is quasi-local with maximal ideal $R_{\alpha} \cap M$, $S$ is quasi-local with maximal ideal $S \cap M$.
    Let $P \in \Ass(T)$. Since $R_{\alpha} \cap P = (0)$ for every $\alpha \in \Omega$, we have $S \cap P = (0)$. Now let $i \in \mathbb{N}$ and suppose $Q \in C_i$.  Let $x \in F_S \cap Q$. Then $s_1x = s_2$ for $s_1,s_2 \in S$. Choose $\alpha \in \Omega$ such that $s_1,s_2 \in R_{\alpha}$. Then $x \in F_{R_{\alpha}} \cap Q \subseteq q_iT$.
    %
\end{proof}

A crucial part of our construction will be finding $Ct\mathfrak{q}$-subrings $S$ of $T$ that satisfy the condition that $q_iT \cap S = q_iS$ for all $i \in \mathbb{N}$. Although not all of our $Ct\mathfrak{q}$-subrings will satisfy this condition, we show that, for any given $Ct\mathfrak{q}$-subring $R$, we can find another $Ct\mathfrak{q}$-subring $S$ containing $R$ that does satisfy the condition that $q_iT \cap S = q_iS$ for all $i \in \mathbb{N}$.

\begin{lemma} \label{lemma 6.5}
    Let $(T,M)$, $C_i$, and $q_i$ for $i \in \mathbb{N}$ be as in Definition \ref{Definition 6.1}. Let $(R, R \cap M)$ be a quasi-local subring of $T$ satisfying all conditions to be a $Ct\mathfrak{q}$-subring of $T$ except that $R$ may be finite. Then $S = F_R \cap T$ satisfies all conditions to be a $Ct\mathfrak{q}$-subring of $T$ except that $S$ might be finite. Moreover, $R \subseteq S \subseteq T$ and $q_iT \cap S = q_iS$ for all $i \in \mathbb{N}$. If $R$ is infinite, then $S$ is a $Ct\mathfrak{q}$-subring of $T$ satisfying $|S| = |R|$. 
\end{lemma}

\begin{proof}
    We first show that $S$ is quasi-local with maximal ideal $S \cap M$. In particular, we show that $x \in F_R \cap T$ is a non-unit if and only if $x \in M \cap F_R \cap T$. Let $x \in F_R \cap T$. Suppose $x \notin M$. Then $x$ is a unit in $T$, and so there exists $t \in T$ such that $tx=1$. Now,  $xr_1=r_2$ for $r_1, r_2 \in R$, and so $r_1 = tr_2$. This implies that $t \in F_R$, and so $x$ is a unit in $F_R \cap T$. Thus, if $x \in F_R \cap T$ is a non-unit then $x \in M \cap F_R \cap T$. The other direction follows since if $x \in M$, then $x$ is not a unit. Thus, $S=F_R\cap T$ is quasi-local with maximal ideal $S \cap M$.

    Note that $R \subseteq S \subseteq T$ and since $R$ contains $q_i$ for all $i \in \mathbb{N}$, so does $S$. If $R$ is finite, then so is $S$ and if $R$ is infinite, then $|S| = |R|$. As a consequence, $|S|<|T|$. Now let $P \in \Ass(T)$ and suppose $x \in S \cap P$. Then $r_1x = r_2$ for some $r_1,r_2 \in R$ with $r_1 \neq 0$ and so $r_1x \in R \cap P = (0)$. Since $r_1$ is not a zerodivisor of $T$, we have $x = 0$. Therefore, $S \cap P = (0)$. Now let $i \in \mathbb{N}$ and suppose $Q \in C_i$. Let $y \in F_S \cap Q$. Then $y \in F_R \cap Q \subseteq q_iT$. It follows that $S$ satisfies all conditions to be a $Ct\mathfrak{q}$-subring of $T$ except that it might be finite. If $R$ is infinite, then $S$ is infinite, and so it is a $Ct\mathfrak{q}$-subring of $T$ satisfying $|S| = |R|$.
    

    Finally, we show that $q_iT \cap S = q_iS$ for all $i \in \mathbb{N}$. Fix $i \in \mathbb{N}$ and let $x \in q_iT \cap S$. Then, we can write $q_it=x$ for $t \in T$ and $xr_2=r_1$ with $r_1$, $r_2 \in R$. This implies that $q_itr_2=r_1$ and so $t \in S$. Therefore, $x \in q_iS$. It follows that $q_iT \cap S \subseteq q_iS$ for all $i \in \mathbb{N}$. The reverse inclusion is a consequence of the fact that $S\subseteq T$. 
\end{proof}

When we adjoin an element of $T$ to a $Ct\mathfrak{q}$-subring, we want the resulting subring of $T$ to also be a $Ct\mathfrak{q}$-subring. The next lemma shows that there are conditions under which we can do this.

\begin{lemma}\label{adjoining transcendentals}
    Let $(T,M)$, $C_i$, and $q_i$ for $i \in \mathbb{N}$ be as in Definition \ref{Definition 6.1}. Moreover, suppose that for every $i \in \mathbb{N}$ and for every $P \in \Ass (T/q_iT)$, we have that $P \subseteq Q$ for some $Q \in C_i$. Let $(R, R\cap M)$ be a quasi-local subring of $T$ containing $q_i$ for all $i \in \mathbb{N}$ and such that, for all $i \in \mathbb{N}$, $q_iT \cap R = q_iR$. Suppose also that $R$ satisfies conditions $(i), (ii)$, and $(iii)$ from Definition \ref{Definition 6.1}. If $x+Q \in T/Q$ is transcendental over $R/(R\cap Q)$ for all $Q \in \bigcup_{i \in \mathbb{N}}C_i$ then $S = R[x]_{(R[x]\cap M)}$ is a $Ct\mathfrak{q}$-subring of $T$. 
\end{lemma}
\begin{proof}
Our proof follows parts of the proof of Lemma 2.8 in \cite{Chatlos} very closely. First note that, since $R$ contains $q_i$ for all $i \in \mathbb{N}$, so does $S$. Since $x + Q$ is transcendental over $R/(R \cap Q)$ for some prime ideal $Q$ of $T$, we have that $R[x]$ is infinite and so $S$ is infinite.  In addition, if $R$ is finite, then $S$ is countable and if $R$ is infinite then $|S| = |R|$. Since $C_1$ is a nonempty set of prime ideals that are not maximal, dim$T \geq 1$. By Proposition \ref{Yu 2.10}, $T$ is uncountable, and so we have $|S| < |T|$.

Fix $i$ and suppose $P \in \Ass(T)$. We claim $P \subseteq Q$ for some $Q \in C_i$. Let $z \in P$. Then there is a nonzero $y \in T$ such that $zy = 0$.  Since $T$ is Noetherian, there is a nonnegative integer $k$ such that $y \in (q_i)^kT$ and $y \not\in (q_i)^{k+ 1}T$. Hence, $y = q_i^kt$ for some $t \in T$ with $t \not\in q_iT$. Now $zq_i^kt = 0$ and since $q_i$ is a regular element of $T$, we have $zt = 0$. Therefore, $z(t + q_iT) = 0 + q_iT$ with $t + q_iT \neq 0 + q_iT$ and it follows that $z \in P'$ for some $P' \in \Ass(T/q_iT)$. By hypothesis, $P' \subseteq Q$ for some $Q \in C_i.$ 
It follows that $P \subseteq \bigcup_{Q \in C_i}Q$. Now use Proposition \ref{stronger coset avoidance} with $C = C_i$, $D = \{0\}$, and $I = P$ to conclude that 
$P \subseteq Q$ for some $Q \in C_i$ as claimed.

 Let $f \in R[x] \cap P \subseteq R[x] \cap Q$ with $f \neq 0$. Then $f = r_nx^n + \cdots + r_1x + r_0 \in Q$ for some $r_j \in R$ with $r_n \neq 0$. Since $x + Q$ is transcendental over $R/(R \cap Q)$, we have $r_j \in R \cap Q$ for all $j = 1,2, \ldots,n$, and since $R$ satisfies condition $(iii)$ of $Ct\mathfrak{q}$-subrings of $T$, we have $R \cap Q \subseteq q_iT$. Let $m$ be the largest positive integer such that $r_j \in q_i^mT$ for all $j = 1,2, \ldots ,n$. Now, $q_i^mT \cap R = q_i^mR$ by Lemma \ref{factoring}. Thus, $r_j = q_i^mr'_j$ for $r'_j \in R$ with at least one $r'_j \not\in q_iT$, and, since $R \cap Q \subseteq q_iT$ we have that at least one $r'_j$ is not in $Q$. Observe that $f = q_i^m(r'_nx^n + \cdots + r'_1x + r'_0) \in P$ and since $q_i$ is a regular element of $T$, it is not in $P$. Hence $r'_nx^n + \cdots + r'_1x + r'_0 \in P \subseteq Q$. This contradicts that $x+Q \in T/Q$ is transcendental over $R/(R\cap Q)$, and so $f = 0$ and we have that $R[x] \cap P = (0)$. It follows that $R[x]$ satisfies condition $(ii)$ of Definition \ref{Definition 6.1}.

 Now fix $i$ and suppose $Q \in C_i$ and $f/g \in F_{R[x]} \cap Q$ where $f,g \in R[x]$ with $g \neq 0$. Then there is an element $p \in Q$ such that $f = pg$. So there exist $r_j,s_j \in R$ such that $f = r_nx^n + \cdots + r_1x + r_0$, $g = s_mx^m + \cdots + s_1x + s_0$ with $s_j \neq 0$ for some $j$ and $$r_nx^n + \cdots + r_1x + r_0 = p(s_mx^m + \cdots + s_1x + s_0).$$ If $f = 0$ then $f/g \in q_iT$ so suppose $f \neq 0$. Then $r_j \neq 0$ for some $j$.  Let $n'$ be the largest integer such that $r_j \in q_i^{n'}T$ for all $j = 1,2, \ldots ,n$ and let $m'$ be the largest integer such that $s_j \in q_i^{m'}T$ for all $j = 1,2, \ldots ,m$. Then we have $$q_i^{n'}(r'_nx^n + \cdots + r'_1x + r'_0) = pq_i^{m'}(s'_mx^m + \cdots + s'_1x + s'_0)$$ where at least one $r'_j \not\in q_iT$ and at least one $s'_j \not\in q_iT$. If $n' \leq m'$ then $r'_nx^n + \cdots + r'_1x + r'_0 \in Q$, contradicting that $x+Q \in T/Q$ is transcendental over $R/(R\cap Q)$. So $n' > m'$. It follows that $p(s'_mx^m + \cdots + s'_1x + s'_0) \in q_iT$. If $p \not\in q_iT$ then $s'_mx^m + \cdots + s'_1x + s'_0 \in P'$ for some $P' \in \Ass(q_iT)$. By hypothesis, we have $s'_mx^m + \cdots + s'_1x + s'_0 \in Q'$ for some $Q' \in C_i$. Now $x+Q' \in T/Q'$ is transcendental over $R/(R\cap Q')$ and so $s'_j \in R \cap Q'$ for all $j = 1,2, \ldots,m$. Since $R \cap Q' \subseteq q_iT$ we have $s'_j \in q_iT$ for all $j = 1,2, \ldots,m$, a contradiction. It follows that $p \in q_iT$ and we have $F_{R[x]} \cap Q \subseteq q_iT$. In other words, $R[x]$ satisfies condition $(iii)$ of Definition \ref{Definition 6.1}.
By Lemma \ref{localization lemma}, $S = R[x]_{(R[x]\cap M)}$ is a $Ct\mathfrak{q}$-subring of $T$. 
%
\end{proof}


Lemma \ref{adjoining transcendentals} gives us conditions on a subring $R$ of $T$ and an element $x$ of $T$ that ensures that the ring $S = R[x]_{(R[x] \cap M)}$ is a $Ct\mathfrak{q}$-subring of $T$. We rely heavily on this result for our construction.

We now work to ensure that the domain $A$ we construct has completion $T$. Recall from Proposition $\ref{completion proving machine}$ that, to do this, it suffices to construct $A$ so that the map $A \longrightarrow T/M^2$ is onto and $IT \cap A = I$ for every finitely generated ideal $I$ of $A$. Given a prime ideal $J$ that is not in any element of any $C_i$ and given an element $u + J$ of $T/J$, the next result allows us to adjoin an element of $u + J$ to a $Ct\mathfrak{q}$-subring of $T$ to obtain another of $Ct\mathfrak{q}$-subring of $T$. We repeatedly apply this lemma to ensure that the map $A \longrightarrow T/M^2$ is onto. The lemma will also be useful in Section 4 and in Section 5, when we construct $A$ to be quasi-excellent.


\begin{lemma} \label{making it onto}
     Let $(T,M)$, $C_i$, and $q_i$ for $i \in \mathbb{N}$ be as in Definition \ref{Definition 6.1}. Moreover, suppose that for every $i \in \mathbb{N}$ and for every $P \in \Ass (T/q_iT)$, we have that $P \subseteq Q$ for some $Q \in C_i$. Let $(R, R \cap M)$ be a $Ct\mathfrak{q}$-subring of $T$ such that $q_iT \cap R = q_iR$ for all $i \in \mathbb{N}$ and let $u+J \in T/J$ where $J$ is an ideal of $T$ with $J \not\subseteq Q$ for all $Q \in \bigcup_{i \in \mathbb{N}}C_i$. Then there exists a $Ct\mathfrak{q}$-subring $S$ of $T$ that satisfies the following conditions:
    \begin{enumerate}[(i)]
        \item $R \subseteq S \subseteq T$,
        \item $|R|=|S|$,
        \item $q_iT \cap S=q_iS$ for all $i \in \mathbb{N}$,
        \item $u+J$ is in the image of the map $S \rightarrow T/J$, and
        \item If $u \in J$, then $S \cap J \not\subseteq Q$ for all $Q \in \bigcup_{i \in \mathbb{N}}C_i$.
    \end{enumerate}
\end{lemma}
\begin{proof}
    Let $Q \in \bigcup_{i \in \mathbb{N}}C_i$ and let $D_{(Q)}$ be a full set of coset representatives of the cosets $t+Q \in T/Q$ that make $(u+t)+Q$ algebraic over $R/(R\cap Q)$. Note that, since $R$ is infinite, $|D_{(Q)}|\leq|R|$ for each $Q \in \bigcup_{i \in \mathbb{N}}C_i$. Let $D=\bigcup_{Q \in \bigcup_{i \in \mathbb{N}}C_i} D_{(Q)}$ and note that $|D|<|T|$. 
    
   Now use Proposition $\ref{stronger coset avoidance}$ with $C = \bigcup_{i \in \mathbb{N}}C_i$ and $I=J$, to find $x \in J$ such that $x \not\in \bigcup \{r+Q \given r \in D, \ Q \in C \}$. Since $(u+x)+Q$ is transcendental over $R/(R\cap Q)$ for all $Q \in \bigcup_{i \in \mathbb{N}}C_i$ and the hypotheses of Lemma $\ref{adjoining transcendentals}$ are met, $S'=R[u+x]_{(R[u+x]\cap M)}$ is a $Ct\mathfrak{q}$-subring of $T$. Now, use Lemma \ref{lemma 6.5} to obtain a $Ct\mathfrak{q}$-subring $S$ of $T$ such that $|R|=|S'|=|S|$, $S' \subseteq S \subseteq T$, and $q_iT \cap S = q_iS$ for all $i \in \mathbb{N}$.

    Since $S' \subseteq S$, the image of $S$ in $T/J$ contains $(u+x)+J = u+J$. Furthermore, if $u \in J$, then $u+x \in J\cap S$, but since $(u+x)+Q$ is transcendental over $R/R \cap Q$ for each $Q \in \bigcup_{i \in \mathbb{N}}C_i$, we have that $u +x \notin Q$. Therefore, $J \cap S \not\subseteq Q$ for all $Q \in \bigcup_{i \in \mathbb{N}}C_i$.
\end{proof}

We now prove a sequence of lemmas that are fundamental in showing that $IT \cap A = I$ for every finitely generated ideal $I$ of $A$.

\begin{lemma}\label{Q equal q_i}
    Let $(T,M)$, $C_i$, and $q_i$ for $i \in \mathbb{N}$ be as in Definition \ref{Definition 6.1}. Let $(R, R\cap M)$ be a quasi-local subring of $T$ containing $q_i$ for all $i \in \mathbb{N}$. Also suppose that for all $i \in \mathbb{N}$ and for all $Q \in C_i$ we have $q_iT \cap R = q_iR$ and $F_R \cap Q \subseteq q_iT$. Then $Q \cap R =q_iR$ for all $Q \in C_i$. In particular, $q_iR$ is a prime ideal of $R$.
\end{lemma}
\begin{proof}
      Let $i \in \mathbb{N}$ and let $Q \in C_i$. Notice that $q_iR \subseteq q_iT \cap R \subseteq Q \cap R$. Now, if $x \in R \cap Q$ then $x \in R \cap Q \subseteq F_R \cap Q \subseteq q_iT$, and so $x \in q_iT \cap R = q_iR$.  It follows that $Q \cap R = q_iR$.
%
\end{proof}




Many of our results for $Ct\mathfrak{q}$-subrings have analogous statements for $pca$-subrings in \cite{Chatlos}.  We note that the next lemma does not have an analogous result in \cite{Chatlos} and it is a key component in generalizing the main theorem from \cite{Chatlos}. 

\begin{lemma} \label{new generators for I}
     Let $(T,M)$, $C_i$, and $q_i$ for $i \in \mathbb{N}$ be as in Definition \ref{Definition 6.1}. Moreover, suppose that for every $i \in \mathbb{N}$ and for every $P \in \Ass (T/q_iT)$, we have that $P \subseteq Q$ for some $Q \in C_i$. Let $(R, R \cap M)$ be a $Ct\mathfrak{q}$-subring of $T$ such that $q_iT \cap R = q_iR$ for all $i \in \mathbb{N}$ and let $I=(y_1,\dots,y_m)R$ be a finitely generated ideal of $R$. If $I \nsubseteq q_i R$ for all $i \in \mathbb{N}$, then there exists a $Ct\mathfrak{q}$-subring $R'$ of $T$ such that
     \begin{enumerate}[(i)]
         \item $IR' = (y_1,\dots,y_m)R' = (\Tilde{y}, y_2, \dots, y_m)R'$ where $\Tilde{y} \notin q_i R'$ for all $i \in \mathbb{N}$,
         \item $R \subseteq R' \subseteq T$, 
         \item $|R'| = |R|$, and
         \item $q_i T \cap R' = q_i R'$ for all $i \in \mathbb{N}$.
     \end{enumerate}
\end{lemma}
\begin{proof}
    Without loss of generality, we may assume that $y_j \neq 0$ for all $j = 1,2, \ldots,m$. 
    By Lemma \ref{Q equal q_i}, $q_iR$ is a prime ideal of $R$ for all $i \in \mathbb{N}$.
    Thus, by Lemma \ref{krull's intersection}, $y_1$ is contained in at most finitely many  ideals of the form $q_i R$. If $y_1 \not\in q_iR$ for all $i \in \mathbb{N}$ then define $R' = R$ and $\Tilde{y} = y_1$, and observe that the lemma holds in this case. So for the rest of the proof, assume that $y_1 \in q_iR$ for at least one $i \in \mathbb{N}$.

     We will define $\Tilde{y} = y_1 + at$ for some carefully chosen $a \in I$ and $t \in T$. Assume, without loss of generality, that $y_1 \in q_kR$ for all $1 \leq k \leq n$ and if $k > n$ with $y_1 \in q_kR$ then $q_kR = q_jR$ for some $1 \leq j \leq n$. Use the Prime Avoidance Theorem to find $a \in R$ such that $a \in I$ and $a \notin \bigcup_{k=1}^n q_kR$. Note that we have chosen $a \in I$ such that for $i \in \mathbb{N}$, if $y_1 \in q_i R$, then $a \notin q_i R$.

    To choose $t$, let $Q \in \bigcup_{i \in \mathbb{N}}C_i$ and let $D_{(Q)}$ be a full set of coset representatives of the elements $t + Q$ of $T/Q$ that are algebraic over $R / (R \cap Q)$. Let $D = \bigcup_{Q \in \bigcup_{i \in \mathbb{N}}C_i} D_{(Q)}$. Now, $T$ is uncountable by Proposition \ref{Yu 2.10}. Since $|R| < |T|$, we have $|R / (R \cap Q)| < |T|$ and the algebraic closure of $R / (R \cap Q)$ in $T/Q$ also has cardinality smaller than $|T|$. Therefore, for every $Q \in \bigcup_{i \in \mathbb{N}}C_i$, $|D_{(Q)}| < |T|$. Recall that there are countably many $C_i$'s and each $C_i$  has countably many elements. Thus, $|D|<|T|$.  Note that $M \nsubseteq Q$ for all $Q \in \bigcup_{i \in \mathbb{N}}C_i$ because the elements of each $C_i$ are nonmaximal. Apply Proposition \ref{stronger coset avoidance} with $C = \bigcup_{i \in \mathbb{N}}C_i$ and $I = M$ to find $t \in M$ such that $t + Q$ is transcendental over $R / (R \cap Q)$ for all $Q \in \bigcup_{i \in \mathbb{N}}C_i$.
    
    By Lemma \ref{adjoining transcendentals}, $S = R[t]_{(R[t] \cap M)}$ is a $Ct\mathfrak{q}$-subring of $T$.  Note that $R \subseteq S \subseteq T$ and $|S| = |R| < |T|$. If $Q \in C_1$ then, by Lemma \ref{Q equal q_i}, $Q \cap R = q_1R$. By Lemma \ref{transcendentals}, $t$ is transcendental over $R$. We claim that $\Tilde{y} = y_1 + a t \not\in q_i R[t]$ for all $i \in \mathbb{N}$. By way of contradiction, assume that $\Tilde{y} \in q_i R[t]$ for some $i \in \mathbb{N}$. Since $t$ is transcendental over $R$, $y_1+at \in q_i R[t]$ implies that $y_1 \in q_iR$ and $a \in q_iR$. However, this is a contradiction since we chose $a$ such that $y_1 \in q_i R$ implies that $a \notin q_i R$. Therefore, $\Tilde{y} \notin q_i R[t]$ for all $i \in \mathbb{N}$. 
    Now suppose that $i \in \mathbb{N}$ and $f \in q_iT \cap R[t]$. Then, for some $a_j \in R$ we have $f = a_nt^n + \cdots + a_1t + a_0 \in q_iT \subseteq Q$ for all $Q \in C_i$. Since $t + Q$ is transcendental over $R / (R \cap Q)$, we have $a_j \in R \cap Q = q_iR$. It follows that $f \in q_iR[t]$ and so we have $q_iT \cap R[t] = q_iR[t]$.  As a consequence, $\Tilde{y} \notin q_i T$ for all $i \in \mathbb{N}$.

   We now show that $(y_1,\dots,y_m)S = (\Tilde{y},\dots,y_m)S$. Since $\Tilde{y} = y_1 + at$ with $a \in (y_1, \ldots, y_m)S$, we have that $(\Tilde{y},\dots,y_m)S \subseteq (y_1,\dots,y_m)S = IS$.
   Notice that $\Tilde{y} - y_1 = at \in (M \cap S)IS$, and so we have $(\Tilde{y}, y_2, \dots, y_m)S + (M \cap S)IS = IS$. By Nakayama's Lemma, we have $IS = (\Tilde{y}, y_2, \dots, y_m)S$. 

   By Lemma \ref{lemma 6.5}, $R'= F_{S} \cap T$ is a $Ct\mathfrak{q}$-subring of $T$ such that $S \subseteq R' \subseteq T$, $|R'| = |S|$, and and $q_i T \cap R' = q_i R'$ for all $i \in \mathbb{N}$. Note also that, since $IS = (\Tilde{y}, y_2, \dots, y_m)S$, we have $IR' = (y_1,\dots,y_m)R' = (\Tilde{y}, y_2, \dots, y_m)R'$. 
   
   We now claim that $\Tilde{y} \notin q_i R'$ for all $i \in \mathbb{N}$. Suppose on the contrary that $\Tilde{y} \in q_i R'$ for some $i \in \mathbb{N}$. Then $\Tilde{y} \in q_iR' \subseteq q_iT$, a contradiction.  Hence, $R'$ is the desired $Ct\mathfrak{q}$-subring of $T$.
\end{proof}

We are ready to show that, if $I$ is a finitely generated ideal of a $Ct\mathfrak{q}$-subring $R$ of $T$, and $c \in IT\cap R$, then we can find a larger $Ct\mathfrak{q}$-subring $S$ of $T$ with $c\in IS$. The statement of the lemma corresponds to the statement of Lemma 2.9 from \cite{Chatlos}, but for $Ct\mathfrak{q}$-subrings rather than $p$ca-subrings. The first part of the proof of Lemma \ref{closing up ideals} below, in particular, is very close to the first part of the proof of Lemma 2.9 in \cite{Chatlos}.

\begin{lemma} \label{closing up ideals}
 Let $(T,M)$, $C_i$, and $q_i$ for $i \in \mathbb{N}$ be as in Definition \ref{Definition 6.1}. Moreover, suppose that for every $i \in \mathbb{N}$ and for every $P \in \Ass (T/q_iT)$, we have that $P \subseteq Q$ for some $Q \in C_i$. Let $(R, R \cap M)$ be a $Ct\mathfrak{q}$-subring of $T$ such that $q_iT \cap R = q_iR$ for all $i \in \mathbb{N}$. Let $I$ be a finitely generated ideal of $R$ and let $c \in IT \cap R$. Then there exists a $Ct\mathfrak{q}$-subring $S$ of $T$ meeting the following conditions:
    \begin{enumerate}[(i)]
        \item $R \subseteq S \subseteq T$,
        \item $|S| = |R|$,
        \item $c \in IS$, and
        \item $q_iT \cap S = q_iS$ for all $i \in \mathbb{N}$.
    \end{enumerate}
\end{lemma}
\begin{proof}
    We induct on the number of generators of $I$. Suppose $I = aR$. If $a = 0$, then $c = 0$. So, $S = R$ is the desired $Ct\mathfrak{q}$-subring. If $a \neq 0$, then let $c = au$ for some $u \in T$.  We show that $S' = R[u]_{(R[u] \cap M)}$ is a $Ct\mathfrak{q}$-subring satisfying the first three conditions and then we apply Lemma \ref{lemma 6.5} to find $S$ such that all conditions are satisfied. Note that $|R|=|R[u]|$, so $R[u]$ satisfies condition $(i)$ of Definition \ref{Definition 6.1}. Let $P \in \Ass(T)$ and let $f \in R[u] \cap P$. Then $f = r_nu^n + \cdots + r_1u + r_0$ where $r_n, \dots ,r_1, r_0 \in R$. It follows that
    \[a^nf = r_nc^n + \cdots + r_1ca^{n-1} + r_0a^n,\]
    and we see that $a^nf \in R \cap P=(0)$ because by assumption, $R$ is a $Ct \mathfrak{q}$-subring. Since $R$ is a domain, $a$ is not a zero-divisor in $T$. It follows that $f=0$ and $R[u] \cap P=(0)$ for all $P \in \Ass(T)$. Therefore, $R[u]$ satisfies condition $(ii)$ of Definition \ref{Definition 6.1}. 
    Now let $i \in \mathbb{N}$ and suppose $x \in F_{R[u]}\cap Q$ for some $Q \in C_i$. Then, $f_2x=f_1$ for $f_1,f_2 \in R[u]$. Find a positive integer $m$ such that $a^mf_1, a^mf_2 \in R$. Then, $a^mf_2x=a^mf_1$, and so $x \in F_{R}\cap Q \subseteq q_iT$ by the properties of $Ct\mathfrak{q}$-subrings. It follows that $R[u]$ satisfies condition $(iii)$ of Definition \ref{Definition 6.1}. By Lemma $\ref{localization lemma}$, $S'= R[u]_{(R[u] \cap M)}$ is a $Ct\mathfrak{q}$-subring of $T$. Lastly, we apply Lemma $\ref{lemma 6.5}$ to find a $Ct\mathfrak{q}$-subring $S$ of $T$ that satisfies all four properties of the lemma.

    Now suppose that $I = (y_1, \dots,y_m)R$ is generated by $m > 1$ elements, and assume that the statement holds for all ideals with $m - 1$ generators. Assume, without loss of generality, that $y_j \neq 0$ for all $1 \leq j \leq m$. Since $c \in IT$, there exists $t_1, t_2,\dots ,t_m \in T$ such that $c = y_1 t_1 + y_2 t_2 + \cdots + y_m t_m$. By Lemma \ref{Q equal q_i}, $q_iR$ is a prime ideal of $R$ for all $i \in \mathbb{N}$. If $I$ were contained in infinitely many ideals of the form $q_iR$ for $i \in \mathbb{N}$, then $y_1$ would be in infinitely many ideals of the form $q_iR$, violating Lemma \ref{krull's intersection}.  It follows that $I$ is contained in at most finitely many ideals of the form $q_iR$. If $I \not\subseteq q_iR$ for all $i \in \mathbb{N}$, then let $I' = I$. 
 On the other hand, if $I \subseteq q_iR$ for at least one $i \in \mathbb{N}$ then, without loss of generality, let $I$ be contained in $q_1 R, \dots, q_s R$ where, if $i,j \in \{1,2, \ldots ,s\}$, then $q_iR = q_j R$ if and only if $i = j$, and, if $j > s$ and $I \subseteq q_jR$, then $q_jR = q_iR$ for some $i \in \{1,2, \ldots ,s\}$.  Now, $y_1 \in q_1R$ and so, by Lemma \ref{largestk}, there is a largest positive integer $k_1$ such that $y_1 \in q_1^{k_1}R$. Similarly, there is a largest positive integer $k_2$ such that $y_2 \in q_1^{k_2}R$. Continue to find $k_3, \ldots ,k_s$.  Let $\ell_1 = \min{\{k_1,\ldots, k_s\}}$. In a similar manner, define $\ell_2, \ldots ,\ell_s$ using $q_2R, \ldots ,q_sR$.
 Now $c = q_1^{\ell_1} \cdots q_s^{\ell_s} (y_1' t_1 + \cdots + y_m' t_m)$ where $y'_1, \ldots ,y'_m \in R$ with $y_k = q_1^{\ell_1} \cdots q_s^{\ell_s}y_k'$ for all $1 \leq k \leq m$ and, if $j \in \{1,2, \ldots ,s\}$ then there is a $y' \in \{y'_1, \ldots ,y'_m\}$ satisfying $y' \not\in q_jR$. Letting $I' = (y_1', \dots, y_m')$, it follows that $I' \not\subseteq q_jR$ for all $j \in \{1,2, \ldots,s\}$. If $j > s$ and $I' \subseteq q_jR$, then $I \subseteq I' \subseteq q_jR$, a contradiction.  Hence, $I' \not\subseteq q_iR$ for all $i \in \mathbb{N}$.
 
 Let $c' = y_1' t_1 + \cdots + y_m' t_m$ and notice that $c = q_1^{\ell_1} \cdots q_s^{\ell_s} c'$. 
 Now $c \in q_1^{\ell_1}T \cap R = q_1^{\ell_1}R$ by Lemma \ref{factoring} and so $q_1^{\ell_1} \cdots q_s^{\ell_s} c' = q_1^{\ell_1}c_1$ for some $c_1 \in R$. Cancelling, we get that $q_2^{\ell_2} \cdots q_s^{\ell_s}c' = c_1 \in R$. Repeat the argument to cancel $q_2^{\ell_2}, \ldots, q_s^{\ell_s}$ and conclude that $c' \in R.$

   
    To proceed, use Lemma \ref{new generators for I} to find a $Ct\mathfrak{q}$-subring $R'$ of $T$ such that
    \begin{enumerate}[(i)]
        \item $(y'_1,y_2',\dots,y_m')R' = (\Tilde{y},y_2',\dots,y_m')R'$ where $\Tilde{y} \notin  q_i R'$ for all $i \in \mathbb{N}$,
        \item $R \subseteq R' \subseteq T$,
        \item $|R'| = |R|$, and
        \item $q_i T \cap R' = q_i R'$ for all $i \in \mathbb{N}$.
    \end{enumerate}

    Without loss of generality, reorder the generators of $I'R' = (\Tilde{y},y_2',\dots,y_m')R'$ so that $y_2' \notin q_i R'$ for all $i \in \mathbb{N}$. Our goal is now to find a $t \in T$ that allows us to adjoin $t_1 + y_2' t$ to $R'$ without disturbing the $Ct\mathfrak{q}$-subring properties. 

    First note that if $Q \in \bigcup_{i \in \mathbb{N}}C_i$ and  $(t_1 + y_2' t) + Q = (t_1 + y_2' t') + Q$ for $t,t' \in T$, then we have that $y_2'(t - t') \in Q$. By Lemma \ref{Q equal q_i}, $Q \cap R' = q_i R'$ for all $Q \in C_i$. Since $y_2' \notin q_iR'$ for all $i \in \mathbb{N}$, we have $y_2' \notin Q$. Thus, $(t - t') \in Q$, and therefore, $t + Q = t' + Q$. As a result, if $t + Q \neq t' + Q$, then $(t_1 + y_2' t) + Q \neq (t_1 + y_2' t') + Q$.

    Let $Q \in \bigcup_{i \in \mathbb{N}}C_i$ and let $D_{(Q)}$ be a full set of coset representatives of elements $t' + Q$ of $T/Q$ that make $(t_1 + y_2' t') + Q$ algebraic over $R' / (R' \cap Q)$. Let $D = \bigcup_{Q \in \bigcup_{i \in \mathbb{N}}C_i} D_{(Q)}$ and note that $|D| < |T|$. Use Proposition \ref{stronger coset avoidance} with $I = T$ and $C = \bigcup_{i \in \mathbb{N}}C_i$ to find an element $t \in T$ such that $t \notin \bigcup \{r + P | r \in D, P \in C \}$. Let $x = t_1 + y_2' t$. Since $x+Q$ is transcendental over $R'/ (R' \cap Q)$ for all $Q \in \bigcup_{i \in \mathbb{N}}C_i$, $R'_1 = R'[x]_{(R'[x] \cap M)}$ is a $Ct\mathfrak{q}$-subring of $T$ by Lemma \ref{adjoining transcendentals}. Note that $|R'_1| = |R'| = |R|$. Using Lemma \ref{lemma 6.5}, let $R''$ be a $Ct\mathfrak{q}$-subring of $T$ such that $R'_1 \subseteq R'' \subseteq T$, $|R''| = |R'_1| = |R|$, and $q_iT \cap R'' = q_iR''$ for all $i \in \mathbb{N}$.

    We now both add and subtract $y_1' y_2' t$ to see that 
    \begin{align*}
        c' &= y_1' t_1 + y_1' y_2' t - y_1' y_2' t + y_2' t_2 + \cdots + y_m' t_m\\ &= y_1' x + y_2'(t_2 - y_1' t) + y_3' t_3 + \cdots + y_m' t_m.
    \end{align*}

    Let $I'' = (y_2',\dots,y_m')R''$ and $c'' = c' - y_1' x$. Then $c'' \in I''T \cap R''$. Use the induction assumption to find a $Ct\mathfrak{q}$-subring $S$ of $T$ such that $|S| = |R''| = |R|$, $R \subseteq R' \subseteq R'' \subseteq S \subseteq T$, $c'' \in I''S$, and $q_iT \cap S = q_iS$ for all $i \in \mathbb{N}$. Then $c' = y_1' x + c'' \in I'S$ since $x \in R'' \subseteq S$ by construction of $R''$. Now, recall that $c = q_1^{\ell_1}  \cdots q_s^{\ell_s} c'$, so, $c \in (q_1^{\ell_1} \cdots  q_s^{\ell_s})I'S$, which implies that $c \in IS$. It follows that $S$ is the desired $Ct\mathfrak{q}$-subring of $T$.
%
\end{proof}

Lemma \ref{Lemma 6.11} is the analogous version of Lemma 2.11 in \cite{Chatlos}, and so the proofs are very similar. The lemma shows that we can find a $Ct\mathfrak{q}$-subring $S$ of $T$ that satisfies several of our desired properties. In particular, $S$ satisfies the condition that $IT \cap S = I$ for every finitely generated ideal $I$ of $S$.


 Before stating Lemma \ref{Lemma 6.11}, we introduce the following useful definition.

 \begin{definition}
Let $\Omega$ be a well ordered set and let $\alpha \in \Omega$. We define $\gamma(\alpha) = \sup\{\beta \in \Omega \, | \, \beta < \alpha\}$.
 \end{definition}

\begin{lemma} \label{Lemma 6.11}
     Let $(T,M)$, $C_i$, and $q_i$ for $i \in \mathbb{N}$ be as in Definition \ref{Definition 6.1}. Moreover, suppose that for every $i \in \mathbb{N}$ and for every $P \in \Ass (T/q_iT)$, we have that $P \subseteq Q$ for some $Q \in C_i$. Let $(R, R \cap M)$ be a $Ct\mathfrak{q}$-subring of $T$ such that $q_iT \cap R = q_iR$ for all $i \in \mathbb{N}$. Let $J$ be an ideal of $T$ with $J \nsubseteq Q$ for all $Q \in \bigcup_{i \in \mathbb{N}}C_i$, and let $u + J \in T/J$. Then there exists a $Ct\mathfrak{q}$-subring $S$ of $T$ such that
    \begin{enumerate}[(i)]
         \item $R \subseteq S \subseteq T$,
        \item $|R|=|S|$,
        \item $u+J$ is in the image of the map $S \longrightarrow T/J$,
        \item If $u \in J$, then $S \cap J \nsubseteq Q$ for all $Q \in \bigcup_{i \in \mathbb{N}}C_i$, and
        \item For every finitely generated ideal $I$ of $S$, we have $IT \cap S = I$.
    \end{enumerate}
\end{lemma}

\begin{proof} First use Lemma \ref{making it onto} to find a $Ct\mathfrak{q}$-subring $R_0$ of $T$ such that $R \subseteq R_0 \subseteq T$, $|R_0| = |R|$, $q_iT \cap R_0 = q_iR_0$ for all $i \in \mathbb{N}$, $u + J$ is in the image of the map $R_0 \longrightarrow T/J$, and if $u \in J$, then $R_0 \cap J \not\subseteq Q$ for all $Q \in \bigcup_{i \in \mathbb{N}}C_i$. Our final ring $S$ will contain $R_0$ and so $S$ will satisfy conditions $(iii)$ and $(iv)$ automatically.

Let $$\Omega = \{(I,c) \, | \, I \mbox{ is a finitely generated ideal of } R_0 \mbox{ and } c \in IT \cap R_0 \},$$
and note that $|\Omega| = |R_0| = |R|$.  Well order $\Omega$ such that it does not have a maximal element, and let $0$ denote its initial element. We now inductively define an increasing chain of $Ct\mathfrak{q}$-subrings of $T$, one for every element of $\Omega$.  The subrings will satisfy the condition that $|R_{\alpha}| = |R|$ for every $\alpha \in \Omega$ and $q_iT \cap R_{\alpha} = q_iR_{\alpha}$ for every $i \in \mathbb{N}$. Note that $R_0$ has already been defined. Now let $\alpha \in \Omega$ and assume that $R_{\beta}$ has been defined for all $\beta < \alpha$ such that $R_{\beta}$ is a $Ct\mathfrak{q}$-subring of $T$ with $|R_{\beta}| = |R|$, if $\mu, \rho \in \Omega$ with $\mu < \rho < \beta$, then $R_{\mu} \subseteq R_{\rho} \subseteq R_{\beta}$, and $q_iT \cap R_{\beta} = q_iR_{\beta}$ for all $i \in \mathbb{N}$. If $\gamma(\alpha) < \alpha$ then $\gamma(\alpha) = (I,c)$ for some $(I,c) \in \Omega$. Define $R_{\alpha}$ to be the $Ct\mathfrak{q}$-subring of $T$ obtained from Lemma \ref{closing up ideals} such that $R_{\gamma(\alpha)} \subseteq R_{\alpha} \subseteq T$, $|R_{\alpha}| = |R_{\gamma(\alpha)}| = |R|$, $c \in IR_{\alpha}$, and $q_iT \cap R_{\alpha} = q_iR_{\alpha}$ for all $i \in \mathbb{N}$. If $\gamma(\alpha) = \alpha$ then define $R_{\alpha} = \bigcup_{\beta < \alpha}R_{\beta}$. By Lemma \ref{unioning lemma}, $R_{\alpha}$ is a $Ct\mathfrak{q}$-subring of $T$ with $|R_{\alpha}| = |R|$. If $x \in q_i T \cap R_{\alpha}$ for some $i \in \mathbb{N}$, then $x \in R_{\beta}$ for some $\beta < \alpha$ and so $x \in q_iT \cap R_{\beta} = q_iR_{\beta} \subseteq q_iR_{\alpha}$. Therefore, $q_iT \cap R_{\alpha} = q_iR_{\alpha}$ for every $i \in \mathbb{N}$. 

Now let $R_1 = \bigcup_{\alpha \in \Omega} R_{\alpha}$. By Lemma \ref{unioning lemma}, $R_1$ is a $Ct\mathfrak{q}$-subring of $T$ with $|R_1| = |R|$, and by the argument at the end of the previous paragraph, $q_iT \cap R_1 = q_iR_1$ for every $i \in \mathbb{N}$. Let $I$ be a finitely generated ideal of $R_0$ and let $c \in IT \cap R_0$. Then $(I,c) = \gamma(\alpha)$ for some $\alpha \in \Omega$ satisfying $\gamma(\alpha) < \alpha$. By construction, $c \in IR_{\alpha} \subseteq IR_1$.  It follows that $IT \cap R_0 \subseteq IR_1$ for every finitely generated ideal $I$ of $R_0$.

Repeat this construction with $R_0$ replaced by $R_1$ to obtain a $Ct\mathfrak{q}$-subring $R_2$ of $T$ with $|R_2| = |R|$, $q_iT \cap R_2 = q_iR_2$ for every $i \in \mathbb{N}$, and $IT \cap R_1 \subseteq IR_2$ for all finitely generated ideals $I$ of $R_1$. Continue to obtain $R_j$ for every $j \in \mathbb{N}$.  Then, for every $j \in \mathbb{N}$, we have $R_j$ is a $Ct\mathfrak{q}$-subring of $T$ with $|R_j| = |R|$, $q_iT \cap R_j = q_iR_j$ for every $i \in \mathbb{N}$, and $IT \cap R_j \subseteq IR_{j + 1}$ for every finitely generated ideal $I$ of $R_j$.

We claim that $S = \bigcup_{j \in \mathbb{N}}R_j$ is the desired $Ct\mathfrak{q}$-subring of $T$. Note that $R \subseteq S \subseteq T$. By Lemma \ref{unioning lemma}, $S$ is a $Ct\mathfrak{q}$-subring of $T$ with $|S| = |R|$. Let $I$ be a finitely generated ideal of $S$ and let $x \in IT \cap S$.  Let $I = (s_1, \ldots ,s_k)$ where $s_j \in S$. Choose $K \in \mathbb{N}$ so that $x,s_1, \ldots ,s_k \in R_K$. Then $x \in (s_1, \ldots,s_k)T \cap R_K \subseteq (s_1, \ldots,s_k)R_{K + 1} \subseteq IS$ and it follows that $IT \cap S = I$ for every finitely generated ideal $I$ of $S$.
\end{proof}

We are now ready to construct a local domain $A$ containing elements $q_i$ for $i \in \mathbb{N}$ that satisfies the needed conditions to be a precompletion of $T$, namely, the map $A \rightarrow T/M^2$ is onto and $IT\cap A=IA$ for every finitely generated ideal $I$ of $A$. Furthermore, $A$ satisfies the condition that for each $i \in \mathbb{N}$, $q_iA$ is a prime ideal of $A$ and the maximal elements in the formal fiber of $A$ at $q_iA$ are exactly the elements of $C_i$. The proof of Lemma \ref{Lemma 6.12} is largely based on the proof of Lemma 2.12 from \cite{Chatlos}.

\begin{lemma}\label{Lemma 6.12}
Let $(T,M)$, $C_i$, $q_i$, and $\mathfrak{q}$ for $i \in \mathbb{N}$ be as in Definition \ref{Definition 6.1}. Moreover, suppose that for every $i \in \mathbb{N}$ and for every $P \in \Ass (T/q_iT)$, we have that $P \subseteq Q$ for some $Q \in C_i$. Let $\Pi$ denote the prime subring of $T$. Suppose $P \cap \Pi[\mathfrak{q}]=(0)$ for every $P \in \Ass (T)$ and for all $i \in \mathbb{N}$, if $Q \in C_i$ then $F_{\Pi[\mathfrak{q}]}\cap Q \subseteq q_{i}T$. Then there exists a local domain $A \subseteq T$ containing $q_i$ for every $i \in \mathbb{N}$ such that the following conditions hold.
%
\begin{enumerate}[(i)]
    \item $\widehat{A} \cong T$,
    \item For every $i \in \mathbb{N}$, $q_iA$ is a prime ideal of $A$ and the maximal elements of the formal fiber of $A$ at $q_iA$ are exactly the elements of $C_i$, 
    \item If $J$ is an ideal of $T$ satisfying that $J \nsubseteq Q$ for all $Q \in \bigcup_{i \in \mathbb{N}}C_i$, then the map $A \rightarrow T/J$ is onto and $J \cap A \nsubseteq Q$ for all $Q \in \bigcup_{i \in \mathbb{N}}C_i$,
    \item If $P'$ is a nonzero prime ideal of $A$ with $P' \neq q_iA$ for all $i \in \mathbb{N}$, then $T \otimes_A k(P') \cong k(P')$, where $k(P')=A_P'/P'A_P'$.
\end{enumerate}
\end{lemma}
\begin{proof}
By assumption and since $T$ is uncountable, $R''_0 = \Pi[\mathfrak{q}]_{(\Pi[\mathfrak{q}] \cap M)}$ satisfies all conditions to be a $Ct\mathfrak{q}$-subring of $T$ except that it might be finite. By Lemma \ref{lemma 6.5}, $R'_0 = F_{R''_0} \cap T$ satisfies all conditions to be a $Ct\mathfrak{q}$-subring of $T$ except that it might be finite. Moreover, $R''_0 \subseteq R'_0 \subseteq T$ and  $q_iT \cap R'_0  = q_iR'_0$ for all $i \in \mathbb{N}$. Note that $R'_0$ is countable since $R''_0$ is countable. Let $Q \in \bigcup_{i \in \mathbb{N}}C_i$ and let $D_{(Q)}$ be a full set of coset representatives for the cosets $t + Q \in T/Q$ that are algebraic over $R'_0/(R'_0 \cap Q)$. Let $D = \bigcup_{Q \in \bigcup_{i \in \mathbb{N}}C_i}D_{(Q)},$ and note that $|D| < |T|$. Use Proposition \ref{stronger coset avoidance} with $I = T$ and $C = \bigcup_{i \in \mathbb{N}}C_i$ to find $x \in T$ with $x \not\in \bigcup\set{r+P\given r\in D,P\in C}$. Then, for all $Q \in \bigcup_{i \in \mathbb{N}}C_i$, $x + Q \in T/Q$ is transcendental over $R'_0/(R'_0 \cap Q)$.  By Lemma \ref{adjoining transcendentals}, $\Tilde{R}'_0 = R'_0[x]_{(R'_0 \cap M)}$ is a $Ct\mathfrak{q}$-subring of $T$. Note that $\Tilde{R}'_0$ is countable. Now use Lemma \ref{lemma 6.5} to find a $Ct\mathfrak{q}$-subring $S$ of $T$ such that $\Tilde{R}'_0 \subseteq S \subseteq T$, $|S| = |\Tilde{R}'_0|$, and $q_iT \cap S = q_iS$ for all $i \in \mathbb{N}$. Finally, use Lemma \ref{Lemma 6.11} with $J = T$ and $u = 0$ to find a $Ct\mathfrak{q}$-subring $R_0$ of $T$ such that $S \subseteq R_0 \subseteq T$, $R_0$ is countable and, for every finitely generated ideal $I$ of $R_0$, we have $IT \cap R_0 = I$.

Now let $$\Omega = \{u + J \, | \, J \mbox{ is an ideal of } T \mbox{ with } J \not\subseteq Q \mbox{ for all } Q \in \bigcup_{i \in \mathbb{N}}C_i \}$$
Since $T$ is infinite and Noetherian $|\Omega| \leq |T|$. Well order $\Omega$ so that each element has fewer than $|\Omega|$ predecessors, and let $0$ denote the initial element of $\Omega$. We recursively define a chain of $Ct\mathfrak{q}$-subrings of $T$, one for each element of $\Omega$. We have already defined $R_0$. Let $\alpha \in \Omega$ and assume that $R_{\beta}$ has been defined for all $\beta < \alpha$ such that $R_{\beta}$ is a $Ct\mathfrak{q}$-subring of $T$, $|R_{\beta}| \leq \max\{\aleph_0, |\{\mu \in \Omega \, | \, \mu \leq \beta\}|\}$, if $\rho < \mu \leq \beta$, then $R_{\rho} \subseteq R_{\mu}$, and $IT \cap R_{\beta} = I$ for every finitely generated ideal $I$ of $R_{\beta}$. First suppose $\gamma(\alpha) < \alpha$. Then $\gamma(\alpha) = u + J$ for some $u + J \in \Omega$. Define $R_{\alpha}$ to be the $Ct\mathfrak{q}$-subring of $T$ obtained from Lemma \ref{Lemma 6.11} so that $R_{\gamma(\alpha)} \subseteq R_{\alpha}$, $|R_{\alpha}| = |R_{\gamma(\alpha)}|$, $u + J$ is in the image of the map $R_{\alpha} \longrightarrow T/J$, if $u \in J$, then $R_{\alpha} \cap J \not\subseteq Q$ for all $Q \in \bigcup_{i \in \mathbb{N}}C_i$, and for every finitely generated ideal $I$ of $R_{\alpha}$ we have $IT \cap R_{\alpha} = I$. Now suppose that $\gamma(\alpha) = \alpha$. Then define $R_{\alpha} = \bigcup_{\beta < \alpha}R_{\beta}.$ Suppose $I = (a_1, \ldots ,a_k)$ is a finitely generated ideal of $R_{\alpha}$. Then if $x \in IT \cap R_{\alpha}$, we have for some $\beta < \alpha$ that $x \in (a_1, \ldots ,a_k)T \cap R_{\beta} = (a_1, \ldots ,a_k)R_{\beta} \subseteq I$.  It follows that $IT \cap R_{\alpha} = I$ for every finitely generated ideal $I$ of $R_{\alpha}$. Note that $|R_{\alpha}| \leq \max\{\aleph_0, |\{\mu \in \Omega \, | \, \mu \leq \alpha\}|\} < |T|$ and so by Lemma \ref{unioning lemma}, $R_{\alpha}$ is a $Ct\mathfrak{q}$-subring of $T$.

We claim that $A = \bigcup_{\alpha \in \Omega} R_{\alpha}$ is the desired domain. By Lemma \ref{unioning lemma}, $A$ is quasi-local and satisfies conditions $(ii)$ and $(iii)$ of Definition \ref{Definition 6.1}. Since $A \cap P = (0)$ for all associated prime ideals $P$ of $T$, $A$ is a domain. Since $IT \cap R_{\alpha} = I$ for every finitely generated ideal $I$ of $R_{\alpha}$ we can show, using the same argument as in the previous paragraph, that $IT \cap A = I$ for every finitely generated ideal $I$ of $A$. By construction, condition $(iii)$ of the lemma is satisfied. In particular, if $Q \in \bigcup_{i\in \mathbb{N}}C_i$, then $Q \neq M$ and so $M^2 \not\subseteq Q$. It follows that the map $A \longrightarrow T/M^2$ is onto.  By Proposition \ref{completion proving machine}, $\widehat{A} \cong T$.
    
    Now we show conditions $(ii)$ and $(iv)$ of the lemma are satisfied. If $Q \in C_i$, then by Lemma \ref{Q equal q_i}, $Q \cap A = q_iA$, and so $Q$ is in the formal fiber of $A$ at $q_iA$. If $J$ is a prime ideal of $T$ such that $J \not\subseteq Q$ for all $Q \in \bigcup_{i \in \mathbb{N}}C_i$, then, by construction, $A \cap J \not\subseteq Q$ for all $Q \in \bigcup_{i \in \mathbb{N}}C_i$. It follows that $J$ is not in the formal fiber of $q_iA$ for all $i \in \mathbb{N}$. Therefore, for every $i \in \mathbb{N}$, the maximal elements of the formal fiber of $A$ at $q_iA$ are exactly the elements of $C_i$. Now let $P'$ be a nonzero prime ideal of $A$ such that $P' \neq q_iA$ for all $i \in \mathbb{N}$. Let $J=P'T$. First, suppose that $J \subseteq Q$ for some $Q \in \bigcup_{i\in\mathbb{N}}C_i$. Then, $P' \subseteq J \cap A \subseteq Q\cap A = q_iA$. This implies that ht$(P') \leq \mbox{ht} (q_iA)$. But, $A$ is a domain and $q_iA$ is prime, so ht$(q_iA)=1$. Therefore, either $P'=q_iA$ or $P'=(0)$, which is a contradiction. It follows that $J \not\subseteq Q$ for all $Q \in \bigcup_{i \in \mathbb{N}}C_i$. By construction, we have that $A \rightarrow T/J$ is onto. Now, since $J \cap A = P'T \cap A = P'$, the map $A/P' \rightarrow T/J$ is an isomorphism, and so, $A/P'$ is complete. Then it follows that $T \otimes_A k(P') \cong (T/P'T)_{\overline{A-P'}} \cong (A/P')_{\overline{A-P'}} \cong A_{P'}/P'A_{P'} \cong k(P')$.
\end{proof}

\begin{remark}\label{remark}
Assume the setting of Lemma \ref{Lemma 6.12} and consider the map $f$ from the set $$X = \{J \in \Spec(T) \, | \, J \not\subseteq Q \mbox{ for all }Q \in \bigcup_{i \in \mathbb{N}}C_i\}$$ to the set $$Y = \{P \in \Spec(A) \, | \, P \neq (0) \mbox{ and } P \neq q_iA \mbox{ for every } i \in \mathbb{N} \}$$ given by $J \longrightarrow J \cap A$. By Lemma \ref{Lemma 6.12}, and since $T$ is a faithfully flat extension of $A$, we know that $f$ does map elements of $X$ to elements of $Y$ and $f$ is onto.
%
 %
    Now let $J \in X$ and let $P = J \cap A \in Y$. Then the map $A \longrightarrow T/J$ is onto and since $J \cap A = P$ we have that $A/P \cong T/J$. In particular, $A/P$ is complete and so $A/P \cong T/PT$.  
 %
 %
    Letting $A/P$ denote its image in $T/PT$, we have 
    $$(J/PT) \cap (A/P) = (J \cap A)/P = P/P = (0).$$
    But, since $A/P \cong T/PT$, there can only be one ideal $I$ of $T/PT$ such that $I \cap (A/P) = (0)$. Thus, $J/PT = PT/PT$. 
 It follows that $J = PT$ and as a consequence $f$ is injective. Therefore, $f$ is bijective and so there is a one to one correspondence between elements of $X$ and elements of $Y$.
\end{remark}

We are now ready to prove the main theorem of this section. 


\begin{theorem} \label{big theorem}
Let $T$ be a complete local ring and let $\Pi$ denote the prime subring of $T$. For each $i \in \mathbb{N}$, let $C_i$ be a nonempty countable set of nonmaximal pairwise incomparable prime ideals of $T$ and suppose that, if $i \neq j$, then either $C_i = C_j$ or no element of $C_i$ is contained in an element of $C_j$. Then there exists a local domain $A \subseteq T$ with $\widehat{A} \cong T$ and, for all $i \in \mathbb{N}$, there is a nonzero prime element $p_i$ of $A$ such that $C_i$ is exactly the set of maximal elements of the formal fiber of $A$ at $p_iA$ if and only if there exists a set of nonzero elements $\mathfrak{q} = \{q_i\}_{i = 1}^{\infty}$ of $T$ satisfying the following conditions
\begin{enumerate}[(i)]
        \item For $i \in \mathbb{N}$ we have $q_i \in \bigcap_{Q \in C_i}Q$ and, if $C_j \neq C_i$ and $Q' \in C_j$, then $q_i \not\in Q'$,
        \item $P \cap \Pi[\mathfrak{q}]=(0)$ for all $P \in \Ass(T)$,
        \item If $i \in \mathbb{N}$ and $P' \in \Ass(T/q_iT)$, then $P' \subseteq Q$ for some $Q \in C_i$, and
        \item If $i \in \mathbb{N}$ and $Q \in C_i$, then $F_{\Pi[\mathfrak{q}]}\cap Q \subseteq q_iT$ where $F_{\Pi[\mathfrak{q}]}$ is the quotient field of $\Pi[\mathfrak{q}]$.
    \end{enumerate}
\end{theorem}
\begin{proof}
    Suppose there exists a set of nonzero elements $\mathfrak{q} = \{q_i\}_{i = 1}^{\infty}$ of $T$ such that conditions $(i)-(iv)$ hold. Since $P \cap \Pi[\mathfrak{q}] = (0)$ for all $P \in \Ass(T)$, $q_i$ is a regular element of $T$ for all $i \in \mathbb{N}$.  
    By Lemma \ref{Lemma 6.12} the desired $A$ exists with $p_i = q_i$ for every $i \in \mathbb{N}$.
     
    Conversely, suppose $A \subseteq T$ is a local domain with $\widehat{A} \cong T$ and suppose that, for every $i \in \mathbb{N}$, there is a nonzero prime element $p_i$ of $A$ such that $C_i$ is exactly the set of maximal elements of the formal fiber of $A$ at $p_iA$. Let $q_i = p_i$ for every $i \in \mathbb{N}$.  We claim that $\mathfrak{q} = \{q_i\}_{i = 1}^{\infty}$ is the desired set of elements of $T$.
    
    Fix $i \in \mathbb{N}$ and suppose $Q \in C_i$. Then $Q$ is in the formal fiber of $A$ at $q_iA$ and so $Q \cap A = q_iA$. It follows that $q_i \in Q$. Thus, $q_i \in \bigcap_{Q \in C_i}Q$. Now suppose $C_j \neq C_i$ and $Q' \in C_j$. By hypothesis, no element of $C_j$ is contained in an element of $C_i$. Note that $Q' \cap A = q_jA$. If $q_i \in Q'$ then $q_iA \subseteq Q' \cap A = q_jA$ and so $q_iA = q_jA$. Thus $Q'$ is in the formal fiber of $A$ at $q_iA$. Therefore, $Q' \subseteq Q$ for some $Q \in C_i$, a contradiction.  It follows that condition $(i)$ of the theorem holds.
    
    
    To prove condition $(ii)$ is satisfied, observe that, since the extension $A \subseteq \widehat{A} = T$ is faithfully flat, any zerodivisor of $T$ which is in $A$ must be a zerodivisor of $A$. Since $A$ is a domain and $\Pi[\mathfrak{q}] \subseteq A$, we must have that $P \cap \Pi[\mathfrak{q}] = (0)$ for all $P \in \Ass (T)$. 
    
    Now we show condition $(iii)$ holds. Since the completion of $A / (q_iT \cap A) = A / q_iA$ is $T/q_iT$, all zerodivisors of $T/q_iT$ contained in $A/q_iA$ are zero-divisors of $A/q_iA$. But, $A/q_iA$ is a domain since $q_iA$ is prime. 
    Thus if $P'\in \Ass(T/q_iT)$ then $P' \cap A \subseteq q_iT \cap A = q_iA$. Since $q_i \in P'$, we also have $q_iA \subseteq A \cap P'$, which gives us that $P' \cap A = q_iA$. Thus, $P'$ is in the formal fiber of $A$ at $q_iA$, and it follows that $P' \subseteq Q$ for some $Q \in C_i$.

    Finally, to show condition $(iv)$ holds, suppose $i \in \mathbb{N}$ and let $Q \in C_i$. Suppose $x \in F_{\Pi[\mathfrak{q}]} \cap Q$. Then $xg=h$ for some $g,h \in \Pi[\mathfrak{q}] \subseteq A$ with $g \neq 0$. Now $h \in gT \cap A = gA$. Since $P \cap \Pi[\mathfrak{q}] = (0)$ for all $P \in \Ass (T)$, we know that $g$ is not a zero-divisor of $T$. It follows that $x \in A$ and so $x \in Q \cap A=q_iA \subseteq q_iT$.
\end{proof}

We end this section with two examples illustrating Theorem \ref{big theorem}.

\begin{example}\label{Example1}
Let $T = \mathbb{Q}[[x_1,x_2,x_3,x_4,x_5,x_6]]$, $C_1 = \{(x_1,x_2 - \alpha x_3) \, | \, \alpha \in \mathbb{Q} \}$, $C_2 = \{(x_4,x_5 - \alpha x_6) \, | \, \alpha \in \mathbb{Q} \}$, $q_1 = x_1$, $q_2 = x_4$, and if $j > 2$, then let $C_j = C_2$ and $q_j = q_2.$ Then the conditions for Theorem \ref{big theorem} are satisfied.  Therefore, $T$ contains a local domain $A$ with $\widehat{A} \cong T$ and $A$ contains prime elements $p_i$ for every $i \in \mathbb{N}$ such that $C_i$ is exactly the set of maximal elements of the formal fiber of $A$ at $p_iA$. In fact, by the proof of Theorem \ref{big theorem}, it can be arranged so that $p_1 = x_1$, $p_2 = x_4$, and $p_j = x_4$ for $j > 2$.
\end{example}


We use the next lemma to aid with our second example.

\begin{lemma}\label{unitprimeavoidance} (\cite{LoeppRotthaus}, Lemma 16)
Let $(T,M)$ be a complete local ring and let $C$ be a countable set of nonmaximal prime ideals of $T$.  Let $D$ be a countable set of elements of $T$.  Let $y \in M$ such that $y \not\in P$ for all $P \in C$.  Then there exists a unit $t$ of $T$ such that
$$yt \not\in \bigcup \{P + r \, | \,P \in C, \, r \in D\}.$$
\end{lemma}

\begin{example}\label{Example2}
Let $T = \mathbb{Q}[[x_1, \ldots,x_K]]$ for $K \geq 2$. Then $T$ has infinitely many height one prime ideals and, since $T$ is a unique factorization domain, all height one prime ideals are principal. Let $\mathcal{P} = \{p_i\}_{i\in \mathbb{N}}$ be a countable collection of prime elements of $T$ such that $p_iT = p_jT$ if and only if $i = j$. Choose $C_i$ for $i \in \mathbb{N}$ so that each $C_i$ is a nonempty finite set of prime ideals of the form $p_kT$ where $p_k \in \mathcal{P}$ and such that if $i \neq j$ then $C_i \cap C_j = \emptyset$. For $i \in \mathbb{N}$, let $C_i = \{p'_1T, \ldots, p'_sT\}$ where $p'_k \in \mathcal{P}$ for $1 \leq k \leq s$. Define $q'_i = \prod_{k = 1}^{s} p'_k$. 
We now use Lemma \ref{unitprimeavoidance} to define the set $\{q_i\}_{i \in \mathbb{N}}$ so that condition $(iv)$ of Theorem \ref{big theorem} holds. Given $i \in \mathbb{N}$, $q_i$ will be an associate of $q_i'$ and so conditions $(i) - (iii)$ of Theorem \ref{big theorem} will be satisfied. 


First, let $R_0 = \mathbb{Q}$, and note that $R_0 \cap Q = (0)$ for all $Q \in \bigcup_{i \in \mathbb{N}}C_i$. For $j \geq 0$, we use $R_j$ to define $q_{j + 1}$, and then we use $q_{j + 1}$ to define a countable subring $R_{j + 1}$ of $T$. We proceed in this manner to define $q_i$ for every $i \in \mathbb{N}$ and $R_j$ for every $j \geq 0$. We ensure that if $i \in \mathbb{N}$ and $Q \in C_i$ then  then $R_j \cap Q = (0)$ if $i > j$ and $R_j \cap Q = q_iR_j$ if $i \leq j$. Assume that $j \geq 0$ and that $R_j$ and $q_i$ for $i \leq j$ have been defined to satisfy these conditions. We now define $q_{j+ 1}$ and $R_{j + 1}$. Let $X_{j+ 1} = \{Q \in \bigcup_{i \in \mathbb{N}}C_i \, | \, Q \not\in C_{j+ 1}\}$. Note that $q'_{j+ 1} \not\in Q$ for all $Q \in X_{j+ 1}$. For $Q \in X_{j+ 1}$, let $D_{(Q)}$ be a full set of coset representatives of the cosets $t + Q \in T/Q$ that are algebraic over $R_j/(R_j \cap Q)$. Define $D = \bigcup_{Q \in X_{j+ 1}}D_{(Q)}$ and note that $D$ is countable. Use Lemma \ref{unitprimeavoidance} with $C = X_{j+ 1}$ and $y = q'_{j+ 1}$ to find a unit $t_{j+ 1} \in T$ such that $q'_{j+ 1}t_{j+ 1} + Q$ is transcendental over $R_j/(R_j \cap Q)$ for all $Q \in X_{j+ 1}$.  Define $q_{j+ 1} = q'_{j+ 1}t_{j+ 1}$, and define $R_{j + 1} = R_j[q_{j+ 1}]$. To see that, for this choice of $q_{j+ 1}$ and $R_{j + 1}$ our desired properties hold, suppose $i \in \mathbb{N}$ and $Q \in C_i$. Let $f \in R_{j + 1} \cap Q$. Then $f = r_nq_{j+ 1}^n + \cdots + r_1q_{j+ 1} + r_0 \in Q$ for $r_k \in R_j$. First suppose $i \neq j + 1$. Then, since $q_{j+ 1} + Q$ is transcendental over $R_j/(R_j \cap Q)$ for all $Q \in X_{j+ 1}$, we have that $r_k \in Q \cap R_j$ for $1 \leq k \leq n$. If $i > j + 1$ then $Q \cap R_j = (0)$ and so $R_{j + 1} \cap Q = (0)$.  If $i < j + 1,$ then $Q \cap R_j = q_iR_j$ and so $f \in q_iR_{j + 1}.$ It follows that $R_{j + 1} \cap Q = q_iR_{j + 1}$. It remains to consider the case where $i = j + 1$. In this case, $q_{j + 1} \in Q$, and so $r_0 \in R_j \cap Q = (0)$. Thus, $f = q_{j + 1}(r_nq_{j + 1}^{n - 1} + \cdots + r_1) \in q_{j + 1}R_{j + 1}$, and it follows that $R_{j + 1} \cap Q = q_{j + 1}R_{j + 1}$.

We now show that, for this choice of the set $\mathfrak{q} = \{q_i\}_{i \in \mathbb{N}}$, condition $(iv)$ of Theorem \ref{big theorem} is satisfied. Let $i \in \mathbb{N}$ and let $Q \in C_i$. Suppose $x \in F_{\Pi[\mathfrak{q}]} \cap Q$. Then there is a $N \in \mathbb{N}$ such that $x = f/g$ for $f,g \in R_N$ with $g \neq 0$. If $i > N$, then $f = xg \in R_N \cap Q = (0)$ and so $0 = x \in q_iT$ and we have that $F_{\Pi[\mathfrak{q}]} \cap Q \subseteq q_iT.$ Now suppose $i \leq N$. If $f,g \in q_iT$ then $f,g \in R_N \cap Q = q_iR_N$, and so $f = q_if'$ and $g = q_ig'$ for $f',g' \in R_N$. Then we have $x = f'/g'$. If $f',g' \in q_iT$, we can repeat this. Note that this process must stop since if it does not, $g \in \bigcap_{\ell \in \mathbb{N}}(q_i)^{\ell}T = (0)$, a contradiction. Therefore, without loss of generality, we may assume that at least one of $f$ and $g$ is not in $q_iT$. Now, $f = xg \in R_N \cap Q = q_iR_N \subseteq q_iT = (t_i\prod_{k = 1}^sp'_k)T$. Since $f \in q_iT$, we have that $g \not\in q_iT$. Suppose $g \in p'_mT$ for some $m \in \{1,2, \ldots,s\}$. Then $g \in R_N \cap p'_mT = q_iR_N \subseteq q_iT$, a contradiction. It follows that $x \in p'_kT$ for all $k \in \{1,2, \ldots, s\}$, and so $x \in q_iT$. Hence, $F_{\Pi[\mathfrak{q}]} \cap Q \subseteq q_iT.$ 

By Theorem \ref{big theorem}, $T$ contains a local domain $A$ with $\widehat{A} \cong T$ and $A$ contains prime elements $p''_i$ for every $i \in \mathbb{N}$ such that $C_i$ is exactly the set of maximal elements of the formal fiber of $A$ at $p''_iA$. In fact, by the proof of Theorem \ref{big theorem}, it can be arranged so that $p''_i = q_i$, for all $i \in \mathbb{N}$.
\end{example}

\section{Countable Precompletions}

It can be shown using Remark \ref{remark} that the the precompletion constructed in Lemma \ref{Lemma 6.12} is necessarily uncountable. Therefore, we have established that it is possible to control the formal fibers of countably many height one prime ideals for an uncountable precompletion. It is interesting to ask whether we can do the same for a countable one. In this section, we prove a result analogous to Theorem \ref{big theorem} where we require that the domain $A$ be countable. 

Before we begin, consider the following illustrative example. 
 Suppose that $A$ is a countable local domain with dim$(A) = 3$ and let $p$ be a prime element of $A$. Let $T$ be the completion of $A$ with respect to its maximal ideal. 
 The completion of the dimension two domain $A' = A/pA$ is $T' = T/pT$. Note that $A'$ is countable and, by  Proposition \ref{stronger coset avoidance}, $T'$ has uncountably many prime ideals. A nonzero element of $A'$ is contained in only finitely many height one prime ideals of $T'$, and it follows that only countably many height one prime ideals of $T'$ contain nonzero elements of $A'$. Hence, the formal fiber of $A'$ at its zero ideal contains uncountably many height one prime ideals of $T'$. As a result, the formal fiber of $A$ at $pA$ contains uncountably many height two prime ideals of $T$. So, the formal fiber of $A$ at $pA$ cannot contain countably many maximal elements. This example shows that Theorem \ref{big theorem} does not hold if we simply replace the phrase ``Then there exists a local domain..." with the phrase ``Then there exists a countable local domain..." In the countable version, therefore, we weaken the condition that $C_i$ is exactly the set of maximal elements of the formal fiber of $A$ at $p_iA$ to the condition that the elements of $C_i$ are merely contained in the formal fiber of $A$ at $p_iA$.

We now state a lemma about  cardinalities of residue fields of local rings.

\begin{lemma} \label{cardinalities of local rings}
    (\cite{Barrett}, Lemma 2.12) Let $(A,M)$ be a local ring. If $A/M$ is finite, then $A/M^n$ is finite for all $n$. If $A/M$ if infinite, then $|A/M^n| = |A/M|$ for all $n$.
\end{lemma}

For our precompletion $A$ to be countable, it is necessary that $T/M$ is countable. We now prove a lemma analogous to Lemma \ref{Lemma 6.12} in the previous section, where we assume $T/M$ is countable. The proof of Lemma \ref{Lemma 6.12 countable} follows the proof of Lemma \ref{Lemma 6.12} with only minor adjustments.

\begin{lemma}\label{Lemma 6.12 countable}
Let $(T,M)$, $C_i$, $q_i$, and $\mathfrak{q}$ for $i \in \mathbb{N}$ be as in Definition \ref{Definition 6.1} with the added condition that $T/M$ is countable. Moreover, suppose that for every $i \in \mathbb{N}$ and for every $P \in \Ass (T/q_iT)$, we have that $P \subseteq Q$ for some $Q \in C_i$. Let $\Pi$ denote the prime subring of $T$. Suppose $P \cap \Pi[\mathfrak{q}]=(0)$ for every $P \in \Ass (T)$ and for all $i \in \mathbb{N}$, if $Q \in C_i$ then $F_{\Pi[\mathfrak{q}]}\cap Q \subseteq q_{i}T$. Then there exists a countable local domain $A \subseteq T$ containing $q_i$ for every $i \in \mathbb{N}$ such that the following conditions hold. 
\begin{enumerate}[(i)]
    \item $\widehat{A} \cong T$, and
    \item For every $i \in \mathbb{N}$, $q_iA$ is a prime ideal of $A$ and the formal fiber of $A$ at $q_iA$ contains the elements of $C_i$.
\end{enumerate}
\end{lemma}
\begin{proof}
Follow the proof of Lemma \ref{Lemma 6.12} with $\Omega = T/M^2$. By assumption, $T/M$ is countable, so by Lemma \ref{cardinalities of local rings}, $\Omega$ is countable. Note that, since elements in each $C_i$ are nonmaximal, we have that $M^2 \not\subseteq Q$ for all $Q \in \bigcup_{i \in \mathbb{N}}C_i$. Now make the following minor adjustment to the proof of Lemma \ref{Lemma 6.12}.  When constructing the chain of $Ct\mathfrak{q}$-subrings of $T$, if 
$\gamma(\alpha) < \alpha$ then $\gamma(\alpha) = u + M^2$ for some $u + M^2 \in \Omega$. Define $R_{\alpha}$ to be the $Ct\mathfrak{q}$-subring of $T$ obtained from Lemma \ref{Lemma 6.11} so that $|R_{\alpha}| = |R_{\gamma(\alpha)}|$, $u + M^2$ is in the image of the map $R_{\alpha} \longrightarrow T/M^2$,  and for every finitely generated ideal $I$ of $R_{\alpha}$ we have $IT \cap R_{\alpha} = I$. Then by the proof of Lemma \ref{Lemma 6.12}, $A$ is a domain containing $q_i$ for every $i \in \mathbb{N}$, $\widehat{A} \cong T$, $q_iA$ is a prime ideal of $A$, and the formal fiber of $A$ at $q_iA$ contains the elements of $C_i$. Moreover, since $\Omega$ is countable, $A$ is countable.
\end{proof}


We are now ready to state and prove the analog of Theorem \ref{big theorem} where we require $A$ to be countable. 

\begin{theorem} \label{big theorem countable}
Let $(T,M)$ be a complete local ring and let $\Pi$ denote the prime subring of $T$. For each $i \in \mathbb{N}$, let $C_i$ be a nonempty countable set of nonmaximal pairwise incomparable prime ideals of $T$ and suppose that, if $i \neq j$, then either $C_i = C_j$ or no element of $C_i$ is contained in an element of $C_j$. Then, there exists a countable local domain $A \subseteq T$ with $\widehat{A} \cong T$ and, for all $i \in \mathbb{N}$ there is a nonzero prime element $p_i$ of $A$ such that all elements of $C_i$ are in the formal fiber of $A$ at $p_iA$ if and only if there exists a set of nonzero elements $\mathfrak{q} = \{q_i\}_{i = 1}^{\infty}$ of $T$ satisfying the following conditions
    \begin{enumerate}[(i)]
        \item For $i \in \mathbb{N}$ we have $q_i \in \bigcap_{Q \in C_i}Q$ and, if $q_jT \neq q_iT$ and $Q' \in C_j$, then $q_i \not\in Q'$,
        \item $P \cap \Pi[\mathfrak{q}]=(0)$ for all $P \in \Ass(T)$,
        \item If $i \in \mathbb{N}$ and $Q \in C_i$ or $Q \in \Ass(T/q_iT)$, then $F_{\Pi[\mathfrak{q}]}\cap Q \subseteq q_iT$, 
        \item $T/M$ is countable, and
        \item If $q_iT \neq q_jT$ and $P_i \in \Ass(T/q_iT)$, $P_j \in \Ass(T/q_jT)$, then $P_i \not\subseteq P_j$.
    \end{enumerate}
\end{theorem}
\begin{proof}
    Suppose there is a set of nonzero elements $\mathfrak{q} = \{q_i\}_{i = 1}^{\infty}$ of $T$ such that conditions $(i)-(v)$ hold. Since $P \cap \Pi[\mathfrak{q}] = (0)$ for all $P \in \Ass(T)$, $q_i$ is regular in $T$ for all $i \in \mathbb{N}$. For every $i \in \mathbb{N}$, let $X_i$ be the set of maximal elements of $\Ass(T/q_iT)$. Define $$C'_i = C_i \cup \{Q \in C_j \, | \, q_jT = q_iT \}$$ and define $$C''_i = C'_i \cup \{P \in X_i \, | \, P \not\subseteq Q \mbox{ for all } Q \in C'_i\}.$$ Then $T$, $C''_i$, $q_i$, and $\mathfrak{q}$ for $i \in \mathbb{N}$ satisfies the conditions in Definition \ref{Definition 6.1}. By definition of $C''_i$, we have that if $P \in \Ass(T/q_iT)$, then $P \subseteq Q$ for some $Q \in C''_i.$ By Lemma \ref{Lemma 6.12 countable}, the desired domain $A$ exists with $p_i = q_i$ for every $i \in \mathbb{N}$.

    Now suppose there exists a countable local domain $A \subseteq T$ with $\widehat{A} \cong T$ and, for all $i \in \mathbb{N}$ there is a nonzero prime element $p_i$ of $A$ such that all elements of $C_i$ are in the formal fiber of $A$ at $p_iA$. Let $q_i = p_i$ for every $i \in \mathbb{N}$. We claim that $\mathfrak{q} = \{q_i\}_{i = 1}^{\infty}$ is the desired set of elements of $T$. Let $i \in \mathbb{N}$ and suppose $Q \in C_i$. Then $Q$ is in the formal fiber of $A$ at $q_iA$ and so $q_i \in Q$. It follows that $q_i \in \bigcap_{Q \in C_i}Q$. Now suppose $q_jT \neq q_iT$ and let $Q' \in C_j$. If $q_i \in Q'$ then $q_i \in Q' \cap A = q_jA$. Since $q_i$ and $q_j$ are both prime elements of $A$, we have $q_iA = q_jA$ and so $q_iT = q_jT$, a contradiction.  Hence, condition $(i)$ of the theorem holds. The arguments that conditions $(ii)$ and $(iii)$ hold follow from the arguments in the last three paragraphs of the proof of Theorem \ref{big theorem}. Since $\widehat{A} = T$, we have that $|T/M| = |A/(A \cap M)| \leq |A|$. So, $T/M$ is countable and condition $(iv)$ is satisfied. Now suppose $q_iT \neq q_jT$ with $P_i \in \Ass(T/q_iT)$, $P_j \in \Ass(T/q_jT)$ and $P_i \subseteq P_j$.  The completion of the domain $A' = A/q_iA$ is $T' = T/q_iT$. All associated prime ideals of $T'$ must be in the formal fiber of $A'$ at $(0)$. It follows that $A \cap P_i = q_iA$.  Similarly, $A \cap P_j = q_jA$. Therefore, $q_iA \subseteq q_jA$. Since $q_i$ and $q_j$ are nonzero prime elements of $A$, we have $q_iA = q_jA$. Therefore $q_iT = q_jT$, a contradiction. It follows that condition $(v)$ holds.
\end{proof}

Note that the conditions of Theorem \ref{big theorem countable} are satisfied for both Example \ref{Example1} and Example \ref{Example2}, and so, for those examples, there exists a countable local domain $A \subseteq T$ with $\widehat{A} \cong T$ and $A$ contains prime elements $p_i$ for every $i \in \mathbb{N}$ such that all elements of $C_i$ are in the formal fiber of $A$ at $p_iA$.


\section{Excellent and Quasi-excellent Precompletions}
In this section we prove analogous results from Section \ref{uncountable section} where we require the domain $A$ to be quasi-excellent, and where we require the domain $A$ to be excellent. We begin with definitions and several important results about quasi-excellent rings and excellent rings. For the remainder of this paper, if $A$ is a ring and $P$ is a prime ideal of $A$, we use $k(P)$ to denote the field $A_P/PA_P$.

Recall that a Noetherian ring $A$ is said to be quasi-excellent if it satisfies the following two conditions:

\begin{enumerate}
    \item for all $P \in \Spec{(A)}$, the ring $\widehat{A} \otimes_A L$ is regular for every finite field extension $L$ of $k(P)$;
    \item Reg$(B) \subset \Spec(B)$ is open for every finitely generated $A$-algebra $B$.
\end{enumerate}

\noindent Recall, also, that a quasi-excellent ring is said to be excellent if it is universally catenary.

The next lemma gives sufficient conditions for a precompletion of a complete local ring $T$ to be quasi-excellent in the case that $T$ contains the rationals.

\begin{lemma}[\cite{Baily}, Lemma 2.6]\label{Unknown 2.6}
     Let $(T, M)$ be a complete local ring containing the rationals. Given a local subring $(A, A \cap M)$ of $T$ such that $\widehat{A}=T$, the ring $A$ is quasi-excellent if and only if, for every $P \in \Spec(A)$ and every $Q \in \Spec(T)$ satisfying $Q \cap A = P$, the ring $(T/PT)_Q$ is a regular local ring.
\end{lemma}

\noindent Note that, using the language of formal fibers, Lemma \ref{Unknown 2.6} says that $A$ is quasi-excellent if and only if, for every prime ideal $P$ of $A$, the ring $(T/PT)_Q$ is a regular local ring for all elements $Q$ in the formal fiber of $A$ at $P$.

The following two theorems follow from  Theorem 31.6 and Theorem 31.7 in \cite{Matsumura}.

\begin{theorem}\label{Yu 2.7}
    Let $A$ be a local ring such that its completion is equidimensional. Then $A$ is universally catenary.
\end{theorem}

\begin{theorem}\label{converse}
Suppose $A$ is a local domain that is universally catenary.  Then the completion of $A$ is equidimensional.
\end{theorem}

We are now ready to prove the main theorems of this section. In particular, we show the quasi-excellent and excellent analogs to Theorem \ref{big theorem}.

\begin{theorem} \label{glue quasi-excellent}
    Let $T$ be a complete local ring containing the rationals, and let $\Pi \cong \mathbb{Z}$ denote the prime subring of $T$. For each $i \in \mathbb{N}$, let $C_i$ be a nonempty countable set of nonmaximal pairwise incomparable prime ideals of $T$ and suppose that, if $i \neq j$, then either $C_i = C_j$ or no element of $C_i$ is contained in an element of $C_j$. Then, there exists a quasi-excellent local domain $A \subseteq T$ with $\widehat{A} \cong T$ and, for all $i \in \mathbb{N}$, there is a nonzero prime element $p_i$ of $A$, such that $C_i$ is exactly the set of maximal elements of the formal fiber of $A$ at $p_iA$ if and only if there exists a set of nonzero elements $\mathfrak{q} = \{q_i\}_{i = 1}^{\infty}$ of $T$ satisfying the following conditions 
    \begin{enumerate}[(i)]
        
        \item For $i \in \mathbb{N}$ we have $q_i \in \bigcap_{Q \in C_i}Q$ and, if $C_j \neq C_i$ and $Q' \in C_j$, then $q_i \not\in Q'$,
        \item $P \cap \Pi[\mathfrak{q}]=(0)$ for all $P \in \Ass(T)$,
        \item If $i \in \mathbb{N}$ and $P' \in \Ass(T/q_iT)$, then $P' \subseteq Q$ for some $Q \in C_i$,
        \item If $i \in \mathbb{N}$ and $Q \in C_i$, then $F_{\Pi[\mathfrak{q}]}\cap Q \subseteq q_iT$, 
        \item If $Q \in \bigcup_{i \in \mathbb{N}}C_i$ and $q_i \in Q$ then $T_Q/q_iT_Q$ is a regular local ring, and
        \item If $P$ is a prime ideal of $T$ such that there is an $i \in \mathbb{N}$ with $P \subseteq Q$ for some $Q \in C_i$ and $q_i \not\in P$, then $T_P$ is a regular local ring. 
        \end{enumerate}
\end{theorem}
\begin{proof}
    First, suppose $T$ is the completion of a quasi-excellent local domain $A \subseteq T$ with $\widehat{A} \cong T$ and, for all $i \in \mathbb{N}$, there is a nonzero prime element $p_i$ of $A$, such that $C_i$ is exactly the set of maximal elements of the formal fiber of $A$ at $p_iA$. Let $q_i = p_i$ for every $i \in \mathbb{N}$.  We claim that $\mathfrak{q} = \{q_i\}_{i = 1}^{\infty}$ is the desired set of elements of $T$. Conditions $(i)$, $(ii)$, $(iii)$, and $(iv)$ follow from the proof of Theorem \ref{big theorem}. 
  %
 %
    If $Q \in \bigcup_{i \in \mathbb{N}}C_i$ and $q_i \in Q$ then $Q$ is in the formal fiber of $A$ at $q_iA$ and so by Lemma \ref{Unknown 2.6}, the ring $(T/q_iT)_Q$ is a regular local ring.  It follows that condition $(v)$ holds. To show that condition $(vi)$ holds, suppose $P$ is a prime ideal of $T$ such that there is an $i \in \mathbb{N}$ with with $P \subseteq Q$ for some $Q \in C_i$ and $q_i \not\in P$. Then we have $A \cap P \subseteq A \cap Q = q_iA$ and it follows that $A \cap P = (0)$.  Hence $P$ is in the formal fiber of $A$ at $(0)$. By Lemma \ref{Unknown 2.6} the ring $(T/(0)T)_P \cong T_P$ is a regular local ring.
    
  
    Now suppose there exists a set of nonzero elements $\mathfrak{q} = \{q_i\}_{i = 1}^{\infty}$ of $T$ such that conditions $(i)-(vi)$ hold. By Lemma \ref{Lemma 6.12}, there exists a local domain $A \subseteq T$ such that $\widehat{A} \cong T$, and, for every $i \in \mathbb{N}$, $q_i \in A$ and $q_iA$ is a prime ideal of $A$ with the set of maximal elements of the formal fiber of $A$ at $q_iA$ exactly $C_i$. Moreover, if $J$ is an ideal of $T$ satisfying that $J \nsubseteq Q$ for all $Q \in \bigcup_{i \in \mathbb{N}}C_i$, then the map $A \rightarrow T/J$ is onto and $J \cap A \nsubseteq Q$ for all $Q \in \bigcup_{i \in \mathbb{N}}C_i$. Further, if $P'$ is a nonzero prime ideal of $A$ with $P' \neq q_iA$ for all $i \in \mathbb{N}$, then $T \otimes_A k(P') \cong k(P')$, where $k(P')=A_P'/P'A_P'$.
    
    We now show using Lemma \ref{Unknown 2.6} that $A$ is quasi-excellent. First assume that $P'$ is a nonzero prime ideal of $A$ with $P' \neq q_iA$ for all $i \in \mathbb{N}$. Then if $Q$ is in the formal fiber of $A$ at $P'$, the ring $(T/P'T)_{\overline{A - P'}} \cong T \otimes_A k(P')$ is isomorphic to $k(P')$, a field. Since $(T/P'T)_Q$ is a localization of the ring $(T/P'T)_{\overline{A - P'}}$, we have that $(T/P'T)_Q$ is also a field and so it is a regular local ring.
    %
   %
    Now suppose $P' = q_iA$ for some $i \in \mathbb{N}$. Let $Q'$ be in the formal fiber of $A$ at $q_iA$. Then $Q' \subseteq Q$ for some $Q \in C_i$. Since condition $(v)$ holds, we have $(T/q_iT)_Q \cong (T_Q/q_iT_Q)$ is a regular local ring. It follows that $(T/q_iT)_{Q'}$ is a regular local ring.
    %
    %
    Finally, we consider the case where $P' = (0)$. If the prime ideal $J$ of $T$ is in the formal fiber of $A$ at $(0)$, then $J \cap A = (0)$. It follows by condition $(iii)$ of Lemma \ref{Lemma 6.12} that there is some $i \in \mathbb{N}$ such that $J \subseteq Q$ for some $Q \in C_i$. Since $J \cap A = (0)$, we have $q_i \not\in J$. By condition $(vi)$, $T_J$ is a regular local ring, and so we have that $(T/(0)T)_J \cong T_J$ is a regular local ring. By Lemma \ref{Unknown 2.6}, $A$ is quasi-excellent. 
\end{proof}

\begin{theorem} \label{glue excellent}
    Let $T$ be a complete local ring containing the rationals, and let $\Pi \cong \mathbb{Z}$ denote the prime subring of $T$. For each $i \in \mathbb{N}$, let $C_i$ be a nonempty countable set of nonmaximal pairwise incomparable prime ideals of $T$ and suppose that, if $i \neq j$, then either $C_i = C_j$ or no element of $C_i$ is contained in an element of $C_j$. Then, there exists an excellent local domain $A \subseteq T$ with $\widehat{A} \cong T$ and, for all $i \in \mathbb{N}$, there is a nonzero prime element $p_i$ of $A$, such that $C_i$ is exactly the set of maximal elements of the formal fiber of $A$ at $p_iA$ if and only if $T$ is equidimensional and there exists a set of nonzero elements $\mathfrak{q} = \{q_i\}_{i = 1}^{\infty}$ of $T$ satisfying the following conditions 
    \begin{enumerate}[(i)]
        
        \item For $i \in \mathbb{N}$ we have $q_i \in \bigcap_{Q \in C_i}Q$ and, if $C_j \neq C_i$ and $Q' \in C_j$, then $q_i \not\in Q'$,
        \item $P \cap \Pi[\mathfrak{q}]=(0)$ for all $P \in \Ass(T)$,
        \item If $i \in \mathbb{N}$ and $P' \in \Ass(T/q_iT)$, then $P' \subseteq Q$ for some $Q \in C_i$,
        \item If $i \in \mathbb{N}$ and $Q \in C_i$, then $F_{\Pi[\mathfrak{q}]}\cap Q \subseteq q_iT$, 
        \item If $Q \in \bigcup_{i \in \mathbb{N}}C_i$ and $q_i \in Q$ then $T_Q/q_iT_Q$ is a regular local ring, and
        \item If $P$ is a prime ideal of $T$ such that there is an $i \in \mathbb{N}$ with $P \subseteq Q$ for some $Q \in C_i$ and $q_i \not\in P$, then $T_P$ is a regular local ring. 
        \end{enumerate}
\end{theorem}


\begin{proof}
    Suppose there exists an excellent local domain $A \subseteq T$ with $\widehat{A} \cong T$ and, for all $i \in \mathbb{N}$, there is a nonzero prime element $p_i$ of $A$, such that $C_i$ is exactly the set of maximal elements of the formal fiber of $A$ at $p_iA$. Since $A$ is excellent, it is quasi-excellent, and so by Theorem \ref{glue quasi-excellent} there exists a set of nonzero elements $\mathfrak{q} = \{q_i\}_{i = 1}^{\infty}$ of $T$ satisfying conditions $(i) - (vi)$. Since $A$ is universally catenary, we have by Theorem \ref{converse} that $T$ is equidimensional.
    

    Now, suppose $T$ is equidimensional and there exists a set of nonzero elements $\mathfrak{q} = \{q_i\}_{i = 1}^{\infty}$ of $T$ satisfying conditions $(i)-(vi)$. By Theorem \ref{glue quasi-excellent}, there exists a quasi-excellent local domain $A \subseteq T$ with $\widehat{A} \cong T$ and, for all $i \in \mathbb{N}$, there is a nonzero prime element $p_i$ of $A$, such that $C_i$ is exactly the set of maximal elements of the formal fiber of $A$ at $p_iA$. By Theorem \ref{Yu 2.7}, $A$ is universally catenary, and so $A$ is excellent.
    %
\end{proof}

Note that the conditions of Theorem \ref{glue excellent} are satisfied for both Example \ref{Example1} and Example \ref{Example2}, and so for those examples, $T$ contains an excellent local domain $A$ with $\widehat{A} \cong T$ and $A$ contains prime elements $p_i$ for every $i \in \mathbb{N}$ such that $C_i$ is exactly the set of maximal elements of the formal fiber of $A$ at $p_iA$.


\section{Countable Excellent and Quasi-excellent Precompletions}

In this section we prove analogous results to Theorems \ref{glue quasi-excellent} and \ref{glue excellent}, but where the domain $A$ is required to be countable. 
 Remark \ref{remark} can be used to show that the precompletions constructed in the proofs of Theorems \ref{glue quasi-excellent} and \ref{glue excellent} are uncountable. In order to make our quasi-excellent and excellent precompletions countable, then, we need a different approach. The ideas in this section are inspired by results presented in \cite{Teresa}. 
%
%
%
%
Given a complete local ring $T$, our strategy for constructing a countable precompletion of $T$ is to start with a countable $Ct\mathfrak{q}$-subring of $T$ given by Theorem \ref{big theorem countable}. We then inductively build a countable ascending chain of countable $Ct\mathfrak{q}$-subrings of $T$. Each $Ct\mathfrak{q}$-subring in the chain will have completion $T$. In addition, we adjoin generators of carefully chosen prime ideals of $T$ to each of our subrings. The union of this chain will be a countable quasi-excellent $Ct\mathfrak{q}$-subring of $T$ whose completion is $T$.

Before we begin our construction, we state two preliminary results. We note that, if $R$ is a Noetherian ring, we use Sing$(R)$ to denote the set $\{P \in \Spec(R) \, | \, R_P \mbox{ is not a regular local ring}\}$, and if $I$ is an ideal of $R$, we use $V(I)$ to denote the set of prime ideals of $R$ that contain $I$.

\begin{lemma}[\cite{Rotthaus}, Corollary 1.6]\label{Rott 1.6}
    If $R$ is excellent, then Sing$(R)$ is closed in the Zariski topology, i.e., Sing$(R)=V(I)$ for some ideal $I$ of $R$.
\end{lemma}

\begin{lemma}[\cite{Teresa}, Theorem 3.4] \label{Yu 3.4}
    Let $(T,M)$ be a complete local reduced ring and let $(R, R\cap M)$ be a countable local domain with $R\subseteq T$ and $\widehat{R} = T$. Then 
    \begin{align*}
        \Sigma = \bigcup_{P\in \Spec(R)} \set{Q\in \Spec{T}\given Q\in \Min (I) \text{ for $I$ where $Sing(T/PT) = V(I/PT)$}}
    \end{align*}
    is a countable set. Furthermore, for any prime ideal $Q$ in this set, $Q \not\subseteq P$ for all $P \in \Ass(T)$.
\end{lemma}



We are ready to begin our construction. We use the next lemma to adjoin generators of specific prime ideals of $T$ to a $Ct\mathfrak{q}$-subring of $T$ to obtain another $Ct\mathfrak{q}$-subring of $T$. The result is crucial in showing that our final ring is quasi-excellent.

\begin{lemma} \label{adjoin bad qs}
    Let $(T,M)$, $C_i$, $q_i$, and $\mathfrak{q}$ for $i \in \mathbb{N}$ be as in Definition \ref{Definition 6.1}. Moreover, suppose that for every $i \in \mathbb{N}$ and for every $P \in \Ass (T/q_iT)$, we have that $P \subseteq Q$ for some $Q \in C_i$.
    Suppose $(R, R \cap M)$ is a $Ct\mathfrak{q}$-subring of $T$ such that, for all $i \in \mathbb{N}$, $q_iT \cap R = q_iR$. Let $J \in \Spec(T)$ such that $J \nsubseteq P$ for all $P \in \bigcup_{i \in \mathbb{N}}C_i$. Then there exists a $Ct\mathfrak{q}$-subring of $T$, $(R',R' \cap M)$, such that $R \subseteq R'$, $|R'| = |R|$, and $R'$ contains a generating set for $J$.
\end{lemma}
\begin{proof}
    Let $ J = (x_1,\dots,x_n)$. We inductively define a chain of $Ct\mathfrak{q}$-subrings of $T$, $R=R_1 \subseteq R_2 \subseteq \dots \subseteq R_{n+1}$ such that $R_{n+1}$ contains a generating set for $J$. To construct $R_{i+1}$ from $R_i$, we show that there exists an element $\Tilde{x}_1$ of $T$ so that $R_{i+1} = R_i[\Tilde{x}_i]_{(R_i[\Tilde{x}_i]\cap M)}$ is a $Ct\mathfrak{q}$-subring of $T$ and such that we can replace $x_i$ with $\Tilde{x}_i$ in the generating set of $J$. By Proposition \ref{stronger coset avoidance}, there exists $y \in J$ such that $y \notin P$ for all $P \in \bigcup_{i \in \mathbb{N}}C_i$.

    We construct $R_2$ by adjoining a carefully chosen element of $T$ to  $R_1=R$. We find $\alpha_1 \in M$ so that $\Tilde{x}_1=x_1+\alpha_1y$ satisfies $\Tilde{x}_1+P \in T/P$ is transcendental over $R_1/(R_1 \cap P) = R/(R \cap P) \cong R$ for all $P \in \bigcup_{i \in \mathbb{N}}C_i$. To do this, first let $P \in \bigcup_{i \in \mathbb{N}}C_i$ and let $D_{(P)}$ be a full set of coset representatives of the cosets $t+P \in T/P$ that make $x_1+ty+P$ algebraic over $R_1/(R_1 \cap P)$. Let $D_1=\bigcup_{P \in \bigcup_{i \in \mathbb{N}}C_i}D_{(P)}$ and note that $|D_1| \leq |R|$. It follows that $|D_1| < |T|$. Applying Proposition \ref{stronger coset avoidance} with $I = M$, we have that 
    \begin{align*}
        M \nsubseteq \bigcup\{r+P\given r \in D_1, P \in \bigcup_{i \in \mathbb{N}}C_i\},
    \end{align*}
    so there exists $\alpha_1 \in M$ such that $x_1 + \alpha_1y + P$ is transcendental over $R_1/(R_1 \cap P)$ for every $P \in \bigcup_{i \in \mathbb{N}}C_i$. Let $\Tilde{x}_1 = x_1 + \alpha_1y$. By Lemma \ref{adjoining transcendentals}, we have that $R_2 = R_1[\Tilde{x}_1]_{(R_1[\Tilde{x}_1]\cap M)}$ is a $Ct\mathfrak{q}$-subring of $T$.

    We now show that $J = (\Tilde{x}_1, x_2,\dots,x_n)$. Since $\Tilde{x}_1-x_1 \in MJ$, it follows that $(\Tilde{x}_1, x_2, \dots, x_n)+MJ=J$, and thus, by Nakayama's Lemma, $(\Tilde{x}_1, x_2, \dots, x_n)=J$.

    To construct $R_3$, use the above argument with $R_1$ replaced by $R_2$ to find $\alpha_2 \in M$ such that $\Tilde{x}_2 = x_2 + \alpha_2y$ satisfies 
    $R_3 = R_2[\Tilde{x}_2]_{(R_2[\Tilde{x}_2] \cap M)}$ is a $Ct\mathfrak{q}$-subring of $T$. We have that $J = (\Tilde{x}_1, \Tilde{x}_2, x_3, \dots, x_n)$ by an argument similar to the way we argued that $J = (\Tilde{x}_1, x_2, \dots, x_n)$.

    Repeat the above process for each $i = 4,\dots,n+1$ to obtain a chain of $Ct\mathfrak{q}$-subrings of $T$, $R_1 \subseteq \dots \subseteq R_{n+1}$ and $J = (\Tilde{x}_1, \Tilde{x}_2, \dots, \Tilde{x}_n)$. By construction, $\Tilde{x}_i \in R_{i+1}$, and so $R_{n+1}$ contains a generating set for $J$. Note that $|R_{n + 1}| = |R|$. Thus, $R' = R_{n+1}$ is our desired $Ct\mathfrak{q}$-subring of $T$.
\end{proof}

We are now ready to state and prove the main theorems for this section.

\begin{theorem}\label{glue countable quasi}
Let $(T,M)$ be a complete local ring containing the rationals and let $\Pi \cong \mathbb{Z}$ denote the prime subring of $T$. For each $i \in \mathbb{N}$, let $C_i$ be a nonempty countable set of nonmaximal pairwise incomparable prime ideals of $T$ and suppose that, if $i \neq j$, then either $C_i = C_j$ or no element of $C_i$ is contained in an element of $C_j$. Then, there exists a countable quasi-excellent local domain $A \subseteq T$ with $\widehat{A} \cong T$ and, for all $i \in \mathbb{N}$ there is a nonzero prime element $p_i$ of $A$ such that all elements of $C_i$ are in the formal fiber of $A$ at $p_iA$ if and only if there exists a set of nonzero elements $\mathfrak{q} = \{q_i\}_{i = 1}^{\infty}$ of $T$ satisfying the following conditions
    \begin{enumerate}[(i)]
        \item For $i \in \mathbb{N}$ we have $q_i \in \bigcap_{Q \in C_i}Q$ and, if $q_jT \neq q_iT$ and $Q' \in C_j$, then $q_i \not\in Q'$,
        \item $P \cap \Pi[\mathfrak{q}]=(0)$ for all $P \in \Ass(T)$,
        \item If $i \in \mathbb{N}$ and $Q \in C_i$ or $Q \in \Ass(T/q_iT)$, then $F_{\Pi[\mathfrak{q}]}\cap Q \subseteq q_iT$, 
        \item $T/M$ is countable,
        \item If $q_iT \neq q_jT$ and $P_i \in \Ass(T/q_iT)$, $P_j \in \Ass(T/q_jT)$, then $P_i \not\subseteq P_j$, and 
        \item Suppose $\mathcal{C} = \{Q \in \Spec(T) \, | \, Q \in \bigcup_{i \in \mathbb{N}}C_i \mbox{ or } Q \in \Ass(T/q_iT) \mbox{ for some }i \in \mathbb{N} \}$ and $J \in \Spec(T)$ satisfies $J \subseteq Q$ for some $Q \in \mathcal{C}$. If $q_i \in J$ for some $i \in \mathbb{N}$ then $(T/q_iT)_J$ is a regular local ring, while if $q_i \not\in J$ for all $i \in \mathbb{N}$, then $T_J$ is a regular local ring.
    \end{enumerate}
\end{theorem}

\begin{proof}
Suppose there exists a countable quasi-excellent local domain $A \subseteq T$ with $\widehat{A} \cong T$ and, for all $i \in \mathbb{N}$ there is a nonzero prime element $p_i$ of $A$ such that all elements of $C_i$ are in the formal fiber of $A$ at $p_iA$. Let $q_i = p_i$ for every $i \in \mathbb{N}$. By Theorem \ref{big theorem countable}, the first five conditions hold. Now let $J$ be a prime ideal of $T$ and suppose $J \subseteq Q$ for some $Q \in \mathcal{C}$. Then $A \cap J \subseteq A \cap Q = q_iA$ for some $i \in \mathbb{N}$. If $q_i \in J$ then $A \cap J  = q_iA$ and so $J$ is in the formal fiber of $A$ at $q_iA$. By Lemma \ref{Unknown 2.6}, the ring $(T/q_iT)_J$ is a regular local ring.  Now suppose $q_i \not\in J$ for all $i \in \mathbb{N}$.  Then we have $A \cap J =(0)$. Hence $J$ is in the formal fiber of $A$ at $(0)$. By Lemma \ref{Unknown 2.6} the ring $(T/(0)T)_J \cong T_J$ is a regular local ring.

    


Now suppose conditions $(i) - (vi)$ hold. By Theorem \ref{big theorem countable}, there is a countable local domain $R_0 \subseteq T$ with $\widehat{R}_0 \cong T$ and, for all $i \in \mathbb{N}$, there is a nonzero prime element $p_i$ of $A$ such that all elements of $C_i$ are in the formal fiber of $A$ at $p_iA$. In fact, by the proof of Theorem \ref{big theorem countable}, we can choose $p_i = q_i$ for every $i \in \mathbb{N}$. 

 For every $i \in \mathbb{N}$, let $X_i$ be the set of maximal elements of $\Ass(T/q_iT)$. Define $$C'_i = C_i \cup \{Q \in C_j \, | \, q_jT = q_iT \}$$ and define $$C''_i = C'_i \cup \{P \in X_i \, | \, P \not\subseteq Q \mbox{ for all } Q \in C'_i\}.$$ Then $T$, $C''_i$, $q_i$, and $\mathfrak{q}$ for $i \in \mathbb{N}$ satisfies the conditions in Definition \ref{Definition 6.1}. By the definition of $C''_i$, we have that if $P \in \Ass(T/q_iT)$, then $P \subseteq Q$ for some $Q \in C''_i.$ Suppose $i \in \mathbb{N}$ and $Q \in C''_i$. If $x \in F_{R_0} \cap Q$ then $r = xs$ for some $r,s \in R_0$. So $r \in sT \cap R_0 = sR_0$. Since $s$ is not a zerodivisor in $T$ we have that $x \in R_0$. Thus $x \in R_0 \cap Q = q_iR_0 \subseteq q_iT$. It follows that $R_0$ is a $Ct\mathfrak{q}$-subring of $T$. We note that in this case, and for the rest of the proof, when we say that a ring is a $Ct\mathfrak{q}$-subring of $T$, we mean with respect to the $C''_i$'s.

 We now show that $T$ is reduced. Let $P \in \Ass(T)$.  Then there is an $i \in \mathbb{N}$ such that $P \subseteq Q$ for some $Q \in C''_i$. (To see this, recall the second paragraph of the proof of Lemma \ref{adjoining transcendentals}). Since $q_i$ is not a zerodivisor, $q_i \not\in P$. By condition $(vi)$, $T_P$ is a regular local ring. As a consequence, $T$ satisfies Serre's $(R_0)$ condition. In addition, if $P \in \Ass(T)$ then $PT_P \in \Ass(T_P)$. If ht$P > 0$, then $T_P$ is a regular local ring of dimension at least one and depth zero, a contradiction. Thus, $T$ satisfies Serre's $(S_1)$ condition and it follows that $T$ is reduced.
    
    We define a countable ascending chain of countable $Ct\mathfrak{q}$-subrings of $T$ recursively, starting with $R_0$. For each $Ct\mathfrak{q}$-subring in the chain we ensure its completion is $T$. Let $j \geq 1$ and assume we have defined $R_j$ so that it is a countable $Ct\mathfrak{q}$-subring of $T$ with $\widehat{R_j} \cong T$. We now define $R_{j + 1}$. 
    

    By Lemma \ref{Yu 3.4}, the set
    \begin{align*}
    \Sigma = 
        \bigcup_{P \in \Spec(R_j)}\{J \,| \, J \in \Min(I) \ \text{for} \ I \ \text{where Sing}(T/PT)=V(I/PT)\}
    \end{align*}
    is countable. Define $$\Omega = \{J \in \Sigma \, | \, J \not\subseteq Q \mbox{ for all } Q \in \bigcup_{i \in \mathbb{N}}C''_i\},$$ and note that $\Omega$ is countable. Enumerate $\Omega$, $(J_k)_{k \in \mathbb{N}}$. We recursively define a countable ascending chain of countable $Ct\mathfrak{q}$-subrings of $T$, $S_0 \subseteq S_1 \subseteq \cdots$. Let $S_0 = R_j$. For $k \in \mathbb{N}$, we ensure that $S_{k + 1}$ contains a generating set for all ideals in the set $\{J_1, J_2, \ldots , J_k\}$. Assume that $k \geq 0$ and $S_k$ has been defined so that it is a countable $Ct\mathfrak{q}$-subring of $T$, $q_iT \cap S_k = q_iS_k$ for all $i \in \mathbb{N}$, and $S_k$ contains a generating set for all ideals in the set $\{J_1, \ldots ,J_{k - 1}\}$. Let $S'_{k + 1}$ be the countable $Ct\mathfrak{q}$-subring of $T$ obtained from Lemma \ref{adjoin bad qs} so that $S_k \subseteq S'_{k + 1}$ and $S'_{k + 1}$ contains a generating set for $J_k$. Define $S_{k + 1}$ to be the  countable $Ct\mathfrak{q}$-subring of $T$ obtained from Lemma \ref{lemma 6.5} so that $S'_{k + 1} \subseteq S_{k + 1}$, and $q_iT \cap S_{k + 1} = q_iS_{k + 1}$ for all $i \in \mathbb{N}$. Define $R'_{j + 1} = \bigcup_{k = 1}^{\infty}S_k$. By Lemma \ref{unioning lemma}, $R'_{j + 1}$ is a countable $Ct\mathfrak{q}$-subring of $T$, and by construction, $R'_{j + 1}$ contains a generating set for every element of $\Omega$. For every $i \in \mathbb{N}$ and for every $k \in \mathbb{N}$ we have that $q_iT \cap S_k = q_iS_k$, and it follows that $q_iT \cap R'_{j + 1} = q_iR'_{j + 1}$ for every $i \in \mathbb{N}$. Now use Lemma \ref{Lemma 6.11} to obtain a countable $Ct\mathfrak{q}$-subring of $T$, $R_{j + 1}$ such that $R'_{j + 1} \subseteq R_{j + 1}$ and $IT \cap R_{j + 1} = IR_{j + 1}$ for every finitely generated ideal $I$ of $R_{j + 1}$. Since $R_{j + 1}$ contains $R_0$ and the map $R_0 \longrightarrow T/M^2$ is onto, we have the the map $R_{j + 1} \longrightarrow T/M^2$ is onto.  By Proposition \ref{completion proving machine}, $\widehat{R_{j + 1}} \cong T$.

    Define $A = \bigcup_{j=0}^{\infty}R_j$. Then $A$ is a countable $Ct\mathfrak{q}$-subring of $T$. Since the completion of $R_j$ is $T$ for all $j \in \mathbb{N}$, we have that $IT \cap R_j = I$ for every finitely generated ideal $I$ of $R_j$. It follows (see the proof of Lemma \ref{Lemma 6.12}) that $IT \cap A = I$ for every finitely generated ideal $I$ of $A$. In addtion, we have that the map $A \longrightarrow T/M^2$ is onto.  By Proposition \ref{completion proving machine}, $A$ is Noetherian and $\widehat{A} \cong T$. Since $A$ is a $Ct\mathfrak{q}$-subring of $T$, it is a domain and, if $i \in \mathbb{N}$ and $Q \in C''_i$ then by Lemma \ref{Q equal q_i}, $Q \cap A = q_iA$ and so all elements of $C''_i$ are in the formal fiber of $A$ at $q_iA$.  It follows that all elements of $C_i$ are in the formal fiber of $A$ at $q_iA$.
    


It remains to show that $A$ is quasi-excellent. To do this, we use Lemma \ref{Unknown 2.6}. Let $Q' \in \Spec(T)$ and let $P = Q' \cap A$. Suppose that $(T/PT)_{Q'}$ is not a regular local ring. Then $Q'/PT \in \mbox{Sing}(T/PT) = V(I/PT)$ for some ideal $I$ of $T$. Thus there is a prime ideal $J$ of $T$ that is minimal over $I$ and $J \subseteq Q'$. Then $J/PT \in V(I/PT)$ and so $(T/PT)_J$ is not a regular local ring. Note that $A \cap J = P$. Suppose that $J \subseteq Q$ for some $Q \in \bigcup_{i \in \mathbb{N}}C''_i$. Then $P = A \cap J \subseteq A \cap Q$ and so $P = (0)$ or $P = q_iA$ for some $i \in \mathbb{N}$. If $P = (0)$, then $q_i \not\in J$ for all $i \in \mathbb{N}$ and so by condition $(vi)$, $T_J \cong (T/PT)_J$ is a regular local ring, a contradiction. If $P = q_iA$ for some $i \in \mathbb{N}$, then by condition $(vi)$, $(T/q_iT)_J \cong (T/PT)_J$ is a regular local ring, also a contradiction.  It follows that $J \not\subseteq Q$ for all $Q \in \bigcup_{i \in \mathbb{N}}C''_i$. 

Let $P = (a_1,\ldots, a_m),$ and choose $j \in \mathbb{N}$ so that $a_k \in R_j$ for all $1\leq k \leq m$. Let $P_j = (a_1, \ldots,a_m)R_j$. Then $T/PT = T/P_jT$ and so by construction, $R_{j + 1}$ contains a generating set for $J$. It follows that $A$ contains a generating set for $J$. Therefore, $(T/PT)_J \cong (T/(A \cap J)T)_J \cong (T/JT)_J$, which is a field. This contradicts that $(T/PT)_J$ is not a regular local ring. By Theorem \ref{Unknown 2.6}, it follows that $A$ is quasi-excellent.
\end{proof}

\begin{theorem}\label{glue countable excellent}
Let $(T,M)$ be a complete local ring containing the rationals and let $\Pi \cong \mathbb{Z}$ denote the prime subring of $T$. For each $i \in \mathbb{N}$, let $C_i$ be a nonempty countable set of nonmaximal pairwise incomparable prime ideals of $T$ and suppose that, if $i \neq j$, then either $C_i = C_j$ or no element of $C_i$ is contained in an element of $C_j$. Then, there exists a countable excellent local domain $A \subseteq T$ with $\widehat{A} \cong T$ and, for all $i \in \mathbb{N}$ there is a nonzero prime element $p_i$ of $A$ such that all elements of $C_i$ are in the formal fiber of $A$ at $p_iA$ if and only if $T$ is equidimensional and there exists a set of nonzero elements $\mathfrak{q} = \{q_i\}_{i = 1}^{\infty}$ of $T$ satisfying the following conditions
    \begin{enumerate}[(i)]
        \item For $i \in \mathbb{N}$ we have $q_i \in \bigcap_{Q \in C_i}Q$ and, if $q_jT \neq q_iT$ and $Q' \in C_j$, then $q_i \not\in Q'$,
        \item $P \cap \Pi[\mathfrak{q}]=(0)$ for all $P \in \Ass(T)$,
        \item If $i \in \mathbb{N}$ and $Q \in C_i$ or $Q \in \Ass(T/q_iT)$, then $F_{\Pi[\mathfrak{q}]}\cap Q \subseteq q_iT$, 
        \item $T/M$ is countable,
        \item If $q_iT \neq q_jT$ and $P_i \in \Ass(T/q_iT)$, $P_j \in \Ass(T/q_jT)$, then $P_i \not\subseteq P_j$, and
        \item Suppose $\mathcal{C} = \{Q \in \Spec(T) \, | \, Q \in \bigcup_{i \in \mathbb{N}}C_i \mbox{ or } Q \in \Ass(T/q_iT) \mbox{ for some }i \in \mathbb{N} \}$ and $J \in \Spec(T)$ satisfies $J \subseteq Q$ for some $Q \in \mathcal{C}$. If $q_i \in J$ for some $i \in \mathbb{N}$ then $(T/q_iT)_J$ is a regular local ring, while if $q_i \not\in J$ for all $i \in \mathbb{N}$, then $T_J$ is a regular local ring.
    \end{enumerate}
\end{theorem}

\begin{proof}
Suppose there exists a countable excellent local domain $A \subseteq T$ with $\widehat{A} \cong T$ and, for all $i \in \mathbb{N}$ there is a nonzero prime element $p_i$ of $A$ such that all elements of $C_i$ are in the formal fiber of $A$ at $p_iA$. By Theorem \ref{glue countable quasi}, there is a set of nonzero elements $\mathfrak{q} = \{q_i\}_{i = 1}^{\infty}$ of $T$ that satisfy the six conditions given in the theorem. Since $A$ is universally catenary, $T$ is equidimensional by Theorem \ref{converse}.

Now suppose $T$ is equidimensional and there is a set of nonzero elements $\mathfrak{q} = \{q_i\}_{i = 1}^{\infty}$ of $T$ that satisfy the six conditions given in the theorem. By Theorem \ref{glue countable quasi}, there exists a countable quasi-excellent local domain $A \subseteq T$ with $\widehat{A} \cong T$ and, for all $i \in \mathbb{N}$ there is a nonzero prime element $p_i$ of $A$ such that all elements of $C_i$ are in the formal fiber of $A$ at $p_iA$. By Theorem \ref{Yu 2.7}, since $T$ is equidimensional, $A$ is universally catenary, and so $A$ is excellent.
\end{proof}


Note that the conditions of Theorem \ref{glue countable excellent} are satisfied for both Example \ref{Example1} and Example \ref{Example2} and so for those examples, there exists a countable excellent local domain $A \subseteq T$ with $\widehat{A} \cong T$ and $A$ contains prime elements $p_i$ for every $i \in \mathbb{N}$ such that all elements of $C_i$ are in the formal fiber of $A$ at $p_iA$.

\section{Acknowledgements}

We thank Williams College and the National Science Foundation, via NSF Grant DMS2241623, and NSF Grant DMS1947438 for their generous funding of our research. 

\newpage

\bibliographystyle{plain}
\bibliography{Bibliography}
\end{document}